\documentclass[onefignum,onetabnum]{siamart220329}

\usepackage{amsfonts}
\usepackage{graphicx}
\usepackage{epstopdf}
\usepackage{algorithm,algpseudocode}
\ifpdf
  \DeclareGraphicsExtensions{.eps,.pdf,.png,.jpg}
\else
  \DeclareGraphicsExtensions{.eps}
\fi

\usepackage{subcaption}
\usepackage{multirow}
\usepackage{thm-restate}

\makeatletter
\newcounter{savepagenumber}
\newcommand\mainmatter{%
  \cleardoublepage
  \setcounter{savepagenumber}{\value{page}}
  \@mainmattertrue
  \pagenumbering{arabic}%
}
\newcommand\backmatter{%
  \if@openright
    \cleardoublepage
  \else
    \clearpage
  \fi
  \pagenumbering{roman}%
  \setcounter{page}{\value{savepagenumber}}%
  \@mainmatterfalse
}
\makeatother

\begin{document}

\title{An Efficient Framework for Global Non-Convex Polynomial Optimization over the Hypercube}

\author{Pierre-David Letourneau\thanks{Qualcomm Technologies Inc., \email{pletourn@qti.qualcomm.com}} \and Dalton Jones\thanks{Qualcomm Technologies Inc., \email{daltjone@qti.qualcomm.com}} \and Matthew Morse \thanks{Qualcomm Technologies Inc., \email{mattmors@qti.qualcomm.com}} \and M. Harper Langston \thanks{Qualcomm Technologies Inc., \email{hlangsto@qti.qualcomm.com}}}

\maketitle

\begin{abstract}
    We present a novel, efficient, theoretical, and numerical framework for solving global non-convex polynomial optimization problems. We analytically demonstrate that such problems can be efficiently reformulated using a non-linear objective over a convex set; further, these reformulated problems possess no spurious local minima (i.e., {\em every local minimum is essentially a global minimum}). Through numerical experiments, we achieve polynomial scaling in dimension and degree when computing the optimal value and location of previously intractable high-dimension global polynomial optimization problems. 
\end{abstract}


\section{Introduction}
\label{sec:introduction}
For polynomials $p(x)$ in $D \in \mathbb{N}$ dimensions of degree at most $d \in \mathbb{N}$, consider  global non-convex polynomial optimization problems of the form
\begin{align}
    \label{eq:problemcontinuous}
    \min_{[-1,1]^D} & \;\; p(x).
\end{align}

Based on ideas proposed by Lasserre in~\cite{lasserre2001global}, we have developed a novel theoretical and numerical framework for efficiently solving problems in the form of~\eqref{eq:problemcontinuous}; specifically, we efficiently reformulate such problems into ones involving a non-linear objective over a convex cone of semi-definite matrices in $O(D\cdot d^2)$ dimensions.  Further, our reformulation possesses no spurious local minima.
%(i.e., every local minimum is a global minimum).  

% Our numerical framework for solving the non-convex polynomial optimization problems from~\eqref{eq:problemcontinuous} is based on the augmented Lagrangian and the ``Burer-Monteiro Method'' \cite{burer2003nonlinear} for expressing semi-definite constraints efficiently.

Our main contribution consists of an efficient reformulations of global polynomial optimization problems, where {\em every local minimum is essentially a global minimum}. Importantly, our framework opens the door to the use of pure descent methods to perform global optimization. 

Our reformulation involves a canonical non-linear objective over a cone of semi-definite matrices (Problem \ref{eq:reformulationcontinuous}). Both the original formulation and our proposed reformulation {\em share the same optimal value} (shown in Theorem \ref{thm:equivalence}) and {\em location} (discussed in Section~\ref{sec:optimal_location}).

The remainder of this paper is structured as follows: Previous work is discussed in Section~\ref{subsec:previous_work} while Section~\ref{sec:notation} introduces notation. Our novel reformulation of~\eqref{eq:problemcontinuous} is presented in Section~\ref{sec:theory}, together with the associated theoretical results. Section~\ref{sec:numerical_results} presents numerical results demonstrating the correctness and performance of our framework. Finally, Section~\ref{sec:conclusion} contains conclusions and a discussion of future work. Proofs can be found in Appendix~\ref{sec:proofs}, while details regarding our numerical algorithm may be found in Appendix ~\ref{sec:algorithm}.

\subsection{Previous Work}
\label{subsec:previous_work}
For problems of the form~\eqref{eq:problemcontinuous}, our reformulation and efficient framework are inspired by the polynomial optimization framework introduced by 
Lasserre~\cite{lasserre2001global}, who proposed reformulating~\eqref{eq:problemcontinuous} as a Semi-Definite Program (SDP). Specifically, instead of evaluating polynomials directly, the optimization can be performed by integrating over specific multidimensional measures, the moments of which then replace the monomials of the original polynomial in the objective. The primary advantage of this approach is that the resulting problem is convex; however, rather than provide an exact reformulation, Lasserre's approach~\cite{lasserre2001global} results in a convex relaxation, where the solution provides a {\em lower bound} on the optimal minimum.  The quality of this bound depends on the degree of the Sums-Of-Squares (SOS)~\cite{doi:10.1137/040614141,prajna2002introducing} polynomials used to generate the relaxation; moreover, better relaxations increase the size of the resulting SDP, adversely affecting scaling. In~\cite{lasserre2001global,Parrilo2003} it was demonstrated that any sequence of problems of increasing SOS degree (also known as the ``Lasserre Hierarchy'') will eventually converge to an exact solution in finite time. Although this hierarchy has been found to converge rapidly on some specific problems (see, e.g.,~\cite{de2022convergence}), theoretical estimates indicate that in the worst case, super-exponential sizes are required~\cite{Klerk2011}.

Proofs of convergence and degree estimates for the approach of~\cite{lasserre2001global} rely on results by Pitunar~\cite{putinar1993positive})(``Putinar Positivstellensatz'') and Nie and Schweighofer~\cite{nie2007complexity}. The results of~\cite{putinar1993positive,nie2007complexity} address the problem of expressing positive polynomials over compact semi-algebraic subsets of Euclidean spaces through SOS, or SOS-like, representations of finite degree. While they demonstrate that this can be done in any dimension, these estimates demand a super-exponential SOS degree of the original problem in both dimension and degree; this results in intractable estimates for the cost of solving polynomial optimization problems using~\cite{lasserre2001global}'s framework in the general case. 

By restricting the space of feasible positive polynomials (and measures) to combinations of products of positive one-dimensional polynomials (and product measures), {\em our framework and results overcome the drawbacks of~\cite{lasserre2001global}}. Indeed, Proposition~\ref{prop:measureextension} demonstrates that such polynomials \emph{can} be represented efficiently through a SOS representation (specifically, at a cost \emph{linear} in the dimension and degree of the original problem).  In this restricted case, the moment problem~\footnote{i.e., the problem of extending elements of the dual space of polynomials to positive product measures.} can be characterized and solved efficiently, requiring the knowledge of only linearly many moments in the dimension and degree of the original problem.  Such conclusions are much stronger than the general case considered in~\cite{lasserre2001global} and highlight the efficiency of our novel reformulation.

\subsection{Notation}
\label{sec:notation}

This section introduces the notation used throughout the remainder of the paper. We represent the \emph{dimension} of Problem~\ref{eq:problemcontinuous} as $D \in \mathbb{N}$, and we use $d$ to represent the \emph{problem degree} (i.e., the degree of the objective polynomial). A $D$-dimensional \emph{multi-index} ${\bf n}=(n_1, n_2, ..., n_D)$ is an element of $\mathbb{N}^D$, where $\lvert {\bf n} \rvert = \sum_{i=1}^D {\bf n}_i$ refers to the \emph{multi-index degree}. Unless otherwise stated, a basis of \emph{monomials} of the form,
\begin{equation}
    x^{\bf n} = \prod_{i=1}^D x_i^{{\bf n}_i}
\end{equation}
is used to represent elements of $\mathbb{P}_{D,d}$, \emph{the vector space of polynomials of degree $d$ in dimension $D$}:  
\begin{equation}
\label{eq:P_Dd}
    \mathbb{P}_{D,d} := \left \{ \sum_{\lvert {\bf n} \rvert \leq d} \, p_{\bf n} \, x^{\bf n} : \, p_{\bf n} \in \mathbb{R} \; \forall \; {\bf n} \in \mathbb{N}^D \, \mathrm{s.t.} \, \lvert {\bf n} \rvert \leq d \right \},
\end{equation}
where the short-hand notation 
\begin{equation*}
    \sum_{\lvert {\bf n} \rvert \leq d} a_{\bf n}= \sum_{ {\bf n} \in \{ {
    \bf n} : \lvert {\bf n} \rvert \leq d \} } a_{(n_1, n_2, ..., n_D)}
\end{equation*}
represents summations over all multi-indices of degree less than or equal to $d$.  In particular, note that $\mathbb{P}_{D,d}$ has dimension
\begin{equation}
    \# \{ {\bf n} \in \mathbb{N}^D \, : \, \lvert {\bf n} \rvert \leq d \} = \binom{D+d}{d} \leq \min \left ( D^d, d^D \right ).
\end{equation}
Hence, for a fixed degree $d$, the dimension of the vector space of polynomials space $\mathbb{P}_{D,d}$ grows \emph{polynomially} in the problem dimension $D$. The notation $\mathbb{P}_{D,\infty}$ is used to represent polynomials of any finite degree. The \emph{support of a polynomial} $p \in \mathbb{P}_{D,\infty}$ corresponds to the locations of its nonzero coefficients:
\begin{equation}
    \mathrm{supp}(p) := \{ {\bf n} \in \mathbb{N}^D \, : \, p_{\bf n} \not = 0 \}, 
\end{equation}
and the cardinality (number of elements) of the support is denoted by
$$N(p) := \# \, \mathrm{supp}(p).$$

The set of regular Borel measures over $\mathbb{R}^D$ (\cite{cohn2013measure}) is denoted as $\mathbb{B}_{D}$, and
the \emph{support of a measure} $\mu(\cdot) \in \mathbb{B}_D$ is designated by $\mathrm{supp}(\mu(\cdot))$, representing the smallest set such that $\mu \left ( \mathbb{R}^D \backslash \mathrm{supp}(\mu(\cdot)) \right )  = 0$. 
For a fixed multi-index ${\bf n}$, a $D$-dimensional Borel measure  $\mathbb{B}_{D}$ has an ${\bf n}^{th}$ \emph{moment} defined to be the multi-dimensional Lebesgue integral,
\begin{equation}
    \mu_{\bf n} := \int_{\mathbb{R}^D} x^{\bf n} \, \mathrm{d} \mu(x).
\end{equation}
A \emph{regular product measure}~\footnote{When discussing a measure, parentheses are employed (e.g., $\mu(\cdot)$), whereas the absence of parentheses indicates that we are discussing vectors.} is a regular Borel measure of the form
\begin{equation}
    \mu(\cdot) = \prod_{i=1}^D  \mu_i(\cdot) ,
\end{equation}
where each factor $\mu_i(\cdot)$ belongs to $\mathbb{B}_1$. In this context, the \emph{moments of a product measure} up to degree $d$ correspond to a $D$-tuple of real vectors $( \mu_1, \mu_2, ..., \mu_D ) \in \mathbb{R}^{(2d+1) \times D}$, where each element is defined as
\begin{equation}
    \mu_{i,{\bf n}_i} := \int_{\mathbb{R} } x_i^{{\bf n}_i} \, \mathrm{d} \mu_i (x_i).
\end{equation}

The vector of \emph{moments of the 1D measure} $\mu_i(\cdot)$ is denoted as $\mu_i$.  When a single sub-index is present, $\mu_i \in \mathbb{R}^{2d+1}$ represents the vector of moments, and when two sub-indices are used, $\mu_{i,{\bf n}_i} \in \mathbb{R}$  represents scalar moments (vector elements).\footnote{For a degree $d$, we assume that such vectors contain moments up to degree $2d$. This is necessary for the proper characterization of moment matrices below.}  A set of $L$ moments of product measures is represented by
$$\mu = \{ (\mu_1^{(l)}, ..., \mu_D^{(l)} ) \}_{l} \in \mathbb{R}^{(2d+1) \times D \times L},$$
and we use the symbol $\phi_{\bf n}(\cdot)$ to represent the moments of the linear combination of such measures, i.e.,
\begin{eqnarray}
    \phi_{\bf n}(\mu) := \sum_{l=1}^L  \prod_{i=1}^D \mu^{(l)}_{i,{\bf n}_i}.
\end{eqnarray}

Given a sequence of moments $\mu_i$ of a 1D measure, the \emph{degree-$d$ 1D moment matrix} (or just \emph{moment matrix} when the context is clear) associated with the one-dimensional polynomial $g(x_i) \in \mathbb{P}_{1,d^\prime}$, $0 \leq d^\prime < \infty$ is defined to be the $(d+1) \times (d  + 1) $ matrix $\mathcal{M}_{d} (\mu_i, g(x_i))$ with elements
 \begin{equation}
     [ \mathcal{M}_{d}(\mu_i, g(x)) ]_{{\bf m}_i,{\bf n}_i} = \int g(x_i) \, x_i^{{\bf m}_i + {\bf n}_i} \, \mathrm{d} \mu_i(x_i) = \sum_{k_i=0}^{d^\prime} g_{k_i} \, \mu_{i, (k_i+{\bf m}_i+{\bf n}_i)}
 \end{equation}
for $0 \leq {\bf m}_i, {\bf n}_i \leq d $. When $g(\cdot ) \equiv 1$, we use the short-hand notation
\begin{equation}
    \mathcal{M}_d(\mu_i; 1) = \mathcal{M}_d(\mu_i).
\end{equation}
The construction of such matrices requires the knowledge of only the first $2 d + d^\prime + 1$ moments of $\mu_i (\cdot) $; i.e., $\{ \mu_{i,{\bf n}_i}\}_{{\bf n}_i=0}^{2d + d^\prime} $. We refer to  Proposition~\ref{prop:measureextension}, where it is shown that {\em positive definiteness} of the matrix $\mathcal{M}_d(\mu_i)$ and $\mathcal{M}_{d-1}(\mu_i; 1-x_i^2)$ are necessary and sufficient to ensure the existence of positive measures with such moments. Unless one considers product measures, however, this is true only in the 1D case and generally false in higher dimensions.

\section{Efficient Reformulation for Global Non-Convex Polynomial Optimization Problems}
\label{sec:theory}
This section presents our novel reformulation and theoretical results, demonstrating the correctness and efficiency of the approach for problems of the form in~\eqref{eq:problemcontinuous}. Specifically, given some integer $L$, our reformulation of~\eqref{eq:problemcontinuous} takes the form,
\begin{align}
    \label{eq:reformulationcontinuous}
    \min_{  \mu := \{(\mu^{(l)}_1, ..., \mu^{(l)}_D) \}_{l=1}^L\in \mathbb{R}^{(2d+1)\times D \times L} }&\;\; \sum_{{\bf n} \in \mathrm{supp}(p)} \, p_{\bf n} \, \phi_{\bf n} (\mu) \nonumber \\
    \mathrm{subject\, to} &  \;\; \mathcal{M}_d (\mu^{(l)}_i) \succeq 0,  \\
    &\;\; \mathcal{M}_{d-1} (\mu^{(l)}_i; \, 1-x_i^2) \succeq 0,  \;&  i=1,..., D, l=1,..., L \nonumber \\
    %& \;\; \mu^{(l)}_{i,0} \geq 0, \nonumber \\
    &\;\; \phi_{(0,...,0)}  (\mu)  = 1, \nonumber  \\
    & \;\;  \phi_{\bf n} (\mu)  = \sum_{l=1}^L  \prod_{i=1}^D \mu^{(l)}_{i,{\bf n}_i}. \nonumber
\end{align} 

This reformulation corresponds to a nonlinear objective subject to (convex) semi-definite constraints. Further, because the feasible set lies in a space of size $L \cdot D \cdot (2d + 1) $~\footnote{Note that this space is linear in both degree and dimension}, our reformulation is achieves computational and storage efficiency (see Table~\ref{tab:cost} for more details on computational costs).

Following Proposition \ref{prop:measureextension}, the problem may also be interpreted as the minimization of a linear objective over the set consisting of the normalized convex combinations of $L$ positive product measures (with mass equal to $1$). This interpretation is key for demonstrating how our reformulation of~\ref{eq:reformulationcontinuous} is in fact \emph{equivalent} to~\ref{eq:problemcontinuous} as demonstrated by Theorems~\ref{thm:equivalence}-\ref{thm:globalmin}.

\begin{table}[htbp!]
\small
    \centering
    \begin{tabular}{|c|c|c|c|c|} \hline
        Type & Expression  & Type &  Size & Eval. Cost \\ \hline
        {\bf O} & $\sum_{{\bf n} \in \mathrm{supp}(p)} \, p_{\bf n} \left ( \sum_{l=1}^L  \prod_{i=1}^D \mu^{(l)}_{i,{\bf n}_i} \right )$ & Scalar & $1$ & $ O\left ( N(p) \, L \, D \right )$  \\ \hline
        {\bf C} & $ \mathcal{M}_d \left ( \mu^{(l)}_i \right ) \succeq 0$ & PSD & $(d+1) \times (d+1)$ & $  O\left ( L \, D \, d^2 \right )$  \\ \hline
        {\bf C} & $ \mathcal{M}_{d-1} \left (\mu^{(l)}_i; \, 1-x_i^2 \right )   \succeq 0 $ & PSD & $ d \times d$ & $  O\left ( L \, D \, d^2 \right )$  \\ \hline
        {\bf C} & $  \sum_{l=1}^L \prod_{i=1}^D \mu^{(l)}_{i,0} = 1 $ & Scalar & $1$ & $  O\left ( L \, d \right )$  \\ \hline
    \end{tabular}
    \caption{\footnotesize{Size and evaluation costs of each component (Objective (O) or Constraints (C) Category Type) of the proposed reformulation.}}
    \label{tab:cost}
\end{table}

In the following section, we introduce the theory for how our framework does not suffer from some of the drawbacks of existing approaches; in fact, (1) the dimension of our reformulated problem is fixed and of size at most polynomial in the dimension and degree of the original problem, (2) the optimal value obtained matches that of the original problem exactly, and (3) no hierarchy of problems is required as in Lasserre's formulation~\cite{lasserre2009moments}.

\subsection{Theoretical Results}
\label{sec:reformulation_theory}
The key advantages of our reformulation~\eqref{eq:reformulationcontinuous} are highlighted below in Theorems~\ref{thm:equivalence} and~\ref{thm:globalmin}.  Specifically, Theorem~\ref{thm:equivalence} demonstrates the equivalence between the original and reformulated problem, and Theorem~\ref{thm:globalmin} demonstrates that the reformulated problem may be solved efficiently using local descent techniques only.

In order to prove Theorem~\ref{thm:equivalence}, we rely on Proposition~\ref{prop:measureextension} below, which demonstrates the possibility of efficiently (i.e., at a polynomial cost) extending linear functionals belonging to the dual space $\mathbb{P}_{D,d}^*$ when the latter are in \emph{product form}. This is in contrast to the general case which rules out such extension, a consequence related to the fact that positive polynomials cannot generally be written as a combination of sums-of-squares having a sufficiently small degree (see, e.g., Motzkin~\cite{theodore1965motzkin}). This underlies key differences between our framework and that introduced by Lasserre~\cite{lasserre2009moments}. Indeed, the lack of existence of an efficient representation in Lasserre's formulation leads to reformulations (i.e., semi-definite relaxations) that provide a \emph{lower bound}. To obtain a solution to the original problem under Lasserre's framework, it is necessary to consider a \emph{hierarchy of problems} (the aforementioned ``Lasserre Hierarchy'') of increasing, potentially \emph{exponential} size. 

The theoretical discussion below highlights how our reformulation~\eqref{eq:problemcontinuous} avoids many of the pitfalls of prior approaches with the following key advantages:
\begin{enumerate}
    \item For our reformulation, the dimension is of fixed size that is at most polynomial in the dimension and degree of the original problem;
    \item For our reformulation, we obtain an optimal value equivalent to the original problem;
    \item For our reformulation, we do not require a hierarchy of problems.
\end{enumerate}; 

For the following theorems, proofs are provided in Appendix~\ref{sec:proofs}.

\begin{restatable}{proposition}{measureextension}
\label{prop:measureextension}
    Let $D,d \in \mathbb{N}$ and $( \mu_1, \mu_2, ..., \mu_D ) \in \mathbb{R}^{(2d+1) \times D}$ be such that for each $i = 1, ..., D$, $\mu_{i,0} = 1$, and
    \begin{align*}
        \mathcal{M}_d(\mu_i) &\succeq 0,\\
        \mathcal{M}_{d-1}(\mu_i;  1-x_i^2) &\succeq 0 .
    \end{align*}
    Then, there exists a regular Borel product measure,
    \begin{equation*}
        \mu( \cdot ) := \prod_{i=1}^D  \mu_i(\cdot), 
    \end{equation*}
    supported over $[-1,1]^D$ such that $\mu \left ( [-1,1]^D \right )  = 1$, and
    \begin{equation*}
         \int_{[-1,1]^D} \, x^{\bf n} \, \left (  \prod_{i=1}^D \mu_i(x_i) \right ) \, \mathrm{d} x =  \prod_{i=1}^D \mu_{i,{\bf n}_i}
    \end{equation*}
    for all multi-index ${\bf n} \in \mathbb{N}^D$ such that $0 \leq {\bf n}_i  \leq d$ for all $i$.
\end{restatable}
Proposition~\ref{prop:measureextension} specifically states that the feasible set of Problem~\ref{eq:reformulationcontinuous} corresponds to the set of \emph{normalized convex combination of $L$ product measures} supported over the hypercube $[-1,1]^D$. This, in turn, implies the following results in Theorem~\ref{thm:equivalence} and Theorem~\ref{thm:globalmin} (proofs provided in Appendix~\ref{sec:proofs}).
\begin{restatable}{theorem}{equivalence}
    \label{thm:equivalence}
    Problem \ref{eq:problemcontinuous} and Problem \ref{eq:reformulationcontinuous} share the same global minimum value.
\end{restatable}
\begin{restatable}{theorem}{globalmin}
    \label{thm:globalmin}
    Let $L=2$, and let
    $$ \mu := \left \{  (\mu_1^{(l)}, ..., \mu_D^{(l)})  \right \}_{l=1}^{2} \in \mathbb{R}^{(2d+1) \times D \times 2 } $$ 
    be a feasible point of Problem~\ref{eq:reformulationcontinuous}. This point corresponds to a global minimum if and only if for any $l \in \{0, 1\}$ and $\alpha^{(l)} := \prod_{i=1}^D \mu_{i,0}^{(l)} \not = 0$,
    the point,
    \begin{eqnarray}
        \left \{  \left ( \frac{\mu_1^{(l)}}{\alpha^{(l)}}, \mu_2^{(l)}, ..., \mu_D^{(l)} \right ) ,  (0,...,0)  \right \}
    \end{eqnarray}
    is a local minimum.
    % for any combination of factors $0 < \beta^{(l)}_i $, such that $\prod_{i=1}^D \beta^{(l)}_i = \alpha^{(l)}$.
    % and let $\alpha^{(j)} = \prod_{i=1}^D \mu^{(j)}_{i,0}$
    % for $j \in \{0,1\}$. REWRITE Then, such a point corresponds to a global minimum if and only if $\alpha^{(0)} \not = 0$ and the point,
    % $$\left \{  \frac{ \prod_{i=1}^D\mu^{(0)}_{i,n_i} }{\alpha^{(0)}}  \right \}_{\lvert n \rvert \leq d},$$
    % is a local minimum, or $\alpha^{(1)} \not = 0$ and the point,
    % $$\left \{  \frac{ \prod_{i=1}^D\mu^{(1)}_{i,n_i} }{\alpha^{(1)}}  \right \}_{\lvert n \rvert \leq d},$$
    % is a local minimum.
\end{restatable}
Theorem~\ref{thm:equivalence} implies it is possible to solve Problem~\ref{eq:problemcontinuous} by solving Problem~\ref{eq:reformulationcontinuous}, while Theorem~\ref{thm:globalmin} implies that global optimality can be established using local information only. Indeed, by leveraging Theorem~\ref{thm:globalmin}, a descent algorithm can be devised for computing a global minimum of Problem~\ref{eq:reformulationcontinuous}~\footnote{Complete algorithmic details, pseudocode, numerical implementation details, and further computational considerations are presented in Appendix~\ref{sec:algorithm}}; following the computation of a local minimum, at least one of $\alpha^{(0)}$ or $\alpha^{(1)}$ must be nonzero by construction.  Depending on the nonzero status of $\alpha^{(i)}$, $i\in\{0,1\}$, one may substitute the corresponding point specified in Theorem~\ref{thm:globalmin} and then verify the following conditions:
    \begin{enumerate}
    \item {\bf If no descent path starting from this point exists}, the new point is a local minimum, and we conclude that the original point is a global minimum following Theorem~\ref{thm:globalmin}.
    \item {\bf If a descent path starting from this point exists}, another local minimum with strictly lower value will be computed, and the above process is repeated until converging to a global minimum.
    \end{enumerate}

Even though Theorem~\ref{thm:globalmin} is stated using a combination of only two ($L=2$) product measures, for any $L \geq 2$ the conclusion still holds; the choice of $L=2$ is made for the purpose of simplifying the analysis. In practice, it is common to set $2 \leq L \leq 20$ (see Section~\ref{sec:numerical_results}); moreover, adding more product measures increases the number of degrees of freedom and often improves numerical convergence.

\section{Numerical Results}
\label{sec:numerical_results}

This section presents results demonstrating the correctness and performance of our framework for optimizing polynomials over the hypercube $[-1,1]^D$. For this purpose, we have implemented a descent-based numerical algorithm (described in Appendix~\ref{sec:algorithm}) capable of computing solutions to our proposed reformulation (Problem \ref{eq:reformulationcontinuous}). We emphasize, however, that the choices made in the design and implementation of the algorithm are by no means necessary; indeed, any alternative algorithm capable of solving Problem~\ref{eq:reformulationcontinuous} numerically can be chosen.

Our implementation has been employed in the computation of the global minima of highly non-convex functions with known global optimal value and location as seen in Figures~\ref{fig:sparse} and \ref{fig:dense}. The functions under consideration possess characteristics that make global optimization intractable for traditional techniques; namely, these functions possess non-convexity, exponentially many local minima and exponentially small (in relative volume) basin of attraction as a function of dimension $D$. 

Our numerical experiments were carried out on a Dell 1U server with 2x AMD EPYC 7452 32-core (128 system threads) and 256 GB DRAM.  Experiments were further limited to the use of a \emph{single (1) core}. Timing represents wall-time for solution computation and extraction, excluding I/O and problem setup (i.e., pre-computations). Relative error between a computed value ${\bf x}_{comp}$ and a known value ${\bf x}_{exact}$ refers to the quantity,
\begin{equation}
    \frac{\lvert \lvert {\bf x}_{comp} - {\bf x}_{exact} \rvert \rvert_2 }{ 10^{-15} + \lvert \lvert {\bf x}_{exact} \rvert \rvert_2 }
\end{equation}
where $\lvert \lvert \cdot \rvert \rvert_2$ is the $\ell^2$-norm.

\subsection{Numerical Example 1}
\label{sec:numerical_example_1}
The first example is derived from a class of problems that consists of objective polynomials of the form,\footnote{\footnotesize For this section and the following, $T_n(x)$ represents a Chebyshev polynomial of degree $n$. More details can be found in Section \ref{subsec:polybasis}, Appendix~\ref{sec:algorithm}}
\begin{align*}
    f_D (x) &= \frac{1}{D} \, \sum_{i=1}^D T_2(x_i) - \prod_{i=1}^D T_8(x_i) 
    =  \frac{1}{D} \, \sum_{i=1}^D (2 x_i^2-1)  - \prod_{i=1}^D \cos(8 \, \cos^{-1}(x_i)).
\end{align*}
Figure \ref{fig:sparse} depicts $f_2(\cdot)$ in two dimensions (i.e., $D=2$). 

\begin{figure}[tpbh!]
\begin{subfigure}{.5\textwidth}
  \includegraphics[width=0.85\linewidth]{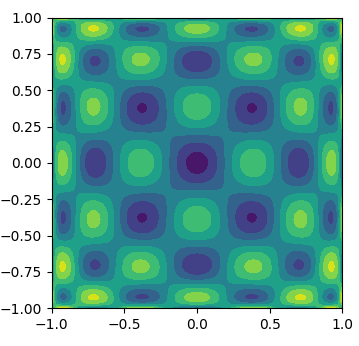}
  \caption{Top view}
  \label{fig:sub1.1}
\end{subfigure}%
\begin{subfigure}{.5\textwidth}
  \includegraphics[width=1.0\linewidth]{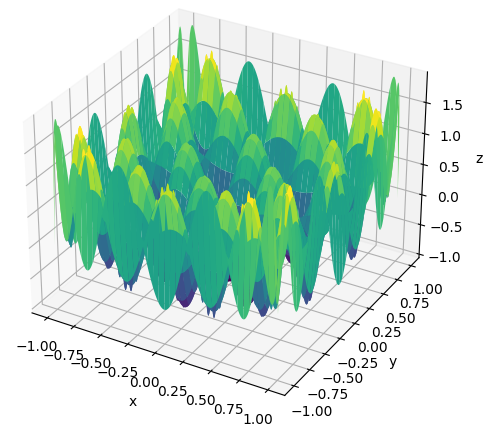}
  \caption{3-D View}
  \label{fig:sub1.2}
\end{subfigure}
\caption{\footnotesize Two-dimensional (2-D) rendering of the  non-convex function $f_2(x)$ in Example~\ref{sec:numerical_example_1}.  For the exponential number of local minima, global optimization of this function would be intractable using traditional techniques.}
\label{fig:sparse}
\end{figure}

For every fixed dimension $D$, the function $f_D (x)$ possess a unique global minimum at the origin (i.e., ${\bf 0} = (0,0,...,0)$) with value equal to
$$f_D^* = f_D({\bf 0}) = -2.0.$$
Each function exhibits more than $3^D$ local minima, and the basin of attraction of the global minimum has volume smaller than $1/5^D$. Together, these characteristics imply that the use of any descent technique is bound to fail. Indeed, any uniformly initialized descent technique will succeed in finding the global minimum with probability 
\emph{at most} $1/5^D$. In dimension $D=10$, this corresponds to a probability of approximately one in $10^{7}$, while in dimension $D=100$, this probability corresponds to approximately one in $10^{70}$, making the problem intractable for descent techniques in general.

Numerical experiments for this example were performed using the algorithm described in Appendix~\ref{sec:algorithm}, together with the parameters shown in Table~\ref{tab:param_family_1}. For this choice of parameters, a relative error smaller than $10^{-2}$ is expected. The relative error for the optimal value and the optimal location, as well as the time required before reaching the stopping criterion, are shown in Figures~\ref{fig:family_1_error} and~\ref{fig:family_1_timing}, respectively for dimension $D$ from $1$ to $250$.

\begin{table}[htbp!]
\small 
    \centering
    \caption{\footnotesize Parameters used for performing numerical experiments related to Example~\ref{sec:numerical_example_1}.}
    \begin{tabular}{|c|c|c|} \hline
        \multirow{4}{*}{Formulation} & Size of moment matrices ($(d+1) \times (d+1)$) & $9$  \\ 
            & Penalty parameter $(\gamma)$  &  $8.0$ \\
            & Burer-Monteiro rank & $9$ (full) \\
            & Number product measures ($L$) & $4$ \\ \hline
        \multirow{3}{*}{Solver (KKT)} & Stopping criterion (relative gradient) & $10^{-2}$ \\
            & Stopping criterion (absolute value) & $10^{-4}$ \\
            & Stopping criterion (absolute feasibility) & $10^{-2}$ \\ \hline
        \multirow{4}{*}{Solver (Primal/L-BFGS)} & Approximation order  &  $100$ \\
                & Wolfe Beta factor ($\beta$) & $0.3$ \\
                & Stopping criterion (relative gradient) & $ 10^{-3} $ \\ 
                & Stopping criterion (absolute value) & $10^{-4}$ \\\hline
    \end{tabular}
    \label{tab:param_family_1}
\end{table}

Figure~\ref{fig:family_1_error} shows that our algorithm succeeds in identifying both the global minimum value and the global optimal location within the expected level of accuracy in every problem instance up to $250$ dimensions.

\begin{figure}[htbp!]
    \centering
    \includegraphics[width=4in]{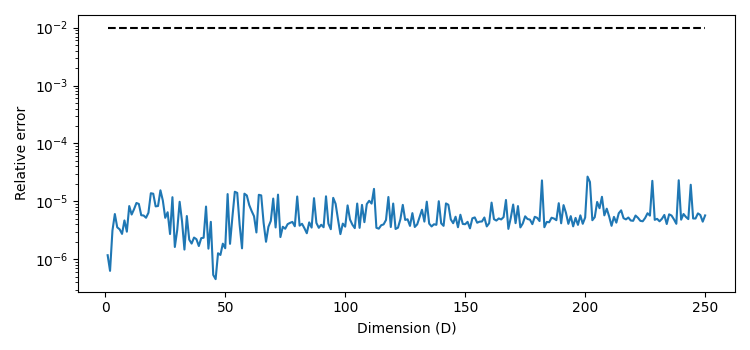}
    \includegraphics[width=4in]{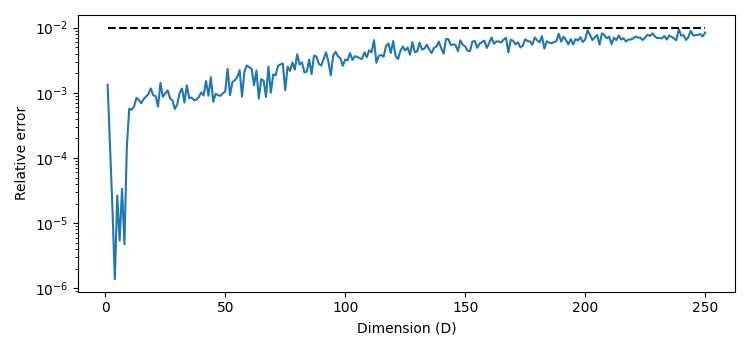}
    \caption{\footnotesize  Relative error vs problem dimension ($D$) for dimension $1$ to $250$ for Example~\ref{sec:numerical_example_1}. {\emph Top}: relative error associated with computed optimal value, {\emph Bottom}: relative error associated with computed optimal location. The error lies below the expected upper bound (black dotted line) of $10^{-2}$ in every case.}
    \label{fig:family_1_error}
\end{figure}

As shown in Figure \ref{fig:family_1_timing}. our algorithm further achieves the desired accuracy in a computationally-efficient manner, exhibiting a highly-advantageous quadratic scaling of the time-to-solution as a function of dimension (i.e. $O(D^2)$ scaling). This is in line with expectations since the support of $f_D$ has $N(f_D) = D+1$ elements, implying that the most expensive component of the algorithm (the evaluation of the objective) carries a cost on the order of $O(LD^2)$ (see Table \ref{tab:cost}).

\begin{figure}[htbp!]
    \centering
    \includegraphics[width=4in]{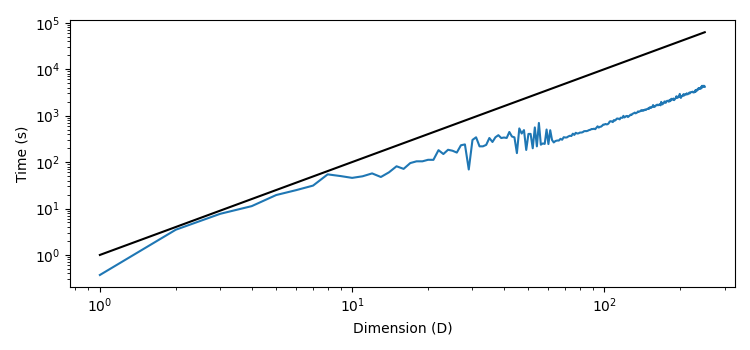}
    \caption{\footnotesize Computation (wall) time (s) vs problem dimension ($D$) for Example~\ref{sec:numerical_example_1}. Our numerical implementation exhibits polynomial (quadratic; (black line)) scaling on problems previously considered intractable.}
    \label{fig:family_1_timing}
\end{figure}

\subsection{Numerical Example 2}
\label{sec:numerical_example_2}
The second numerical example comes from a class of functions of the form,    
\begin{align*}
    g_D (x) &= \frac{1}{D} \, \sum_{i=1}^D T_4(x_i) + \left ( \frac{1}{D} \, \sum_{i=1}^D  T_1(x_i)\right )^3 \\
    &= \frac{1}{D} \, \sum_{i=1}^D \left ( 8x_i^4 - 8x_i^2 + 1 \right ) + \left ( \frac{1}{D} \, \sum_{i=1}^D  x_i \right )^3.
\end{align*}
\begin{figure}[htbp!]
\begin{subfigure}{.5\textwidth}
  \includegraphics[width=0.85\linewidth]{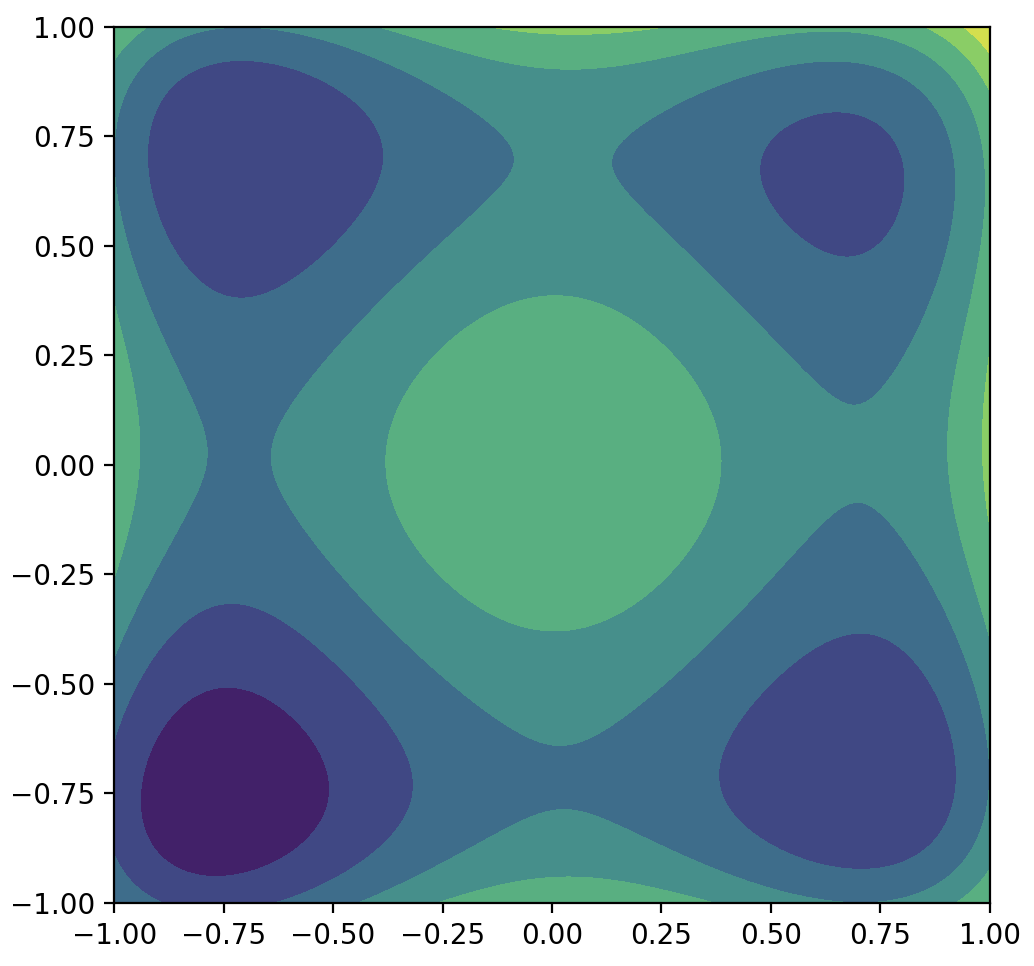}
  \caption{\small Top view}
  \label{fig:sub2.1}
\end{subfigure}%
\begin{subfigure}{.4\textwidth}
  \includegraphics[width=1.25\linewidth]{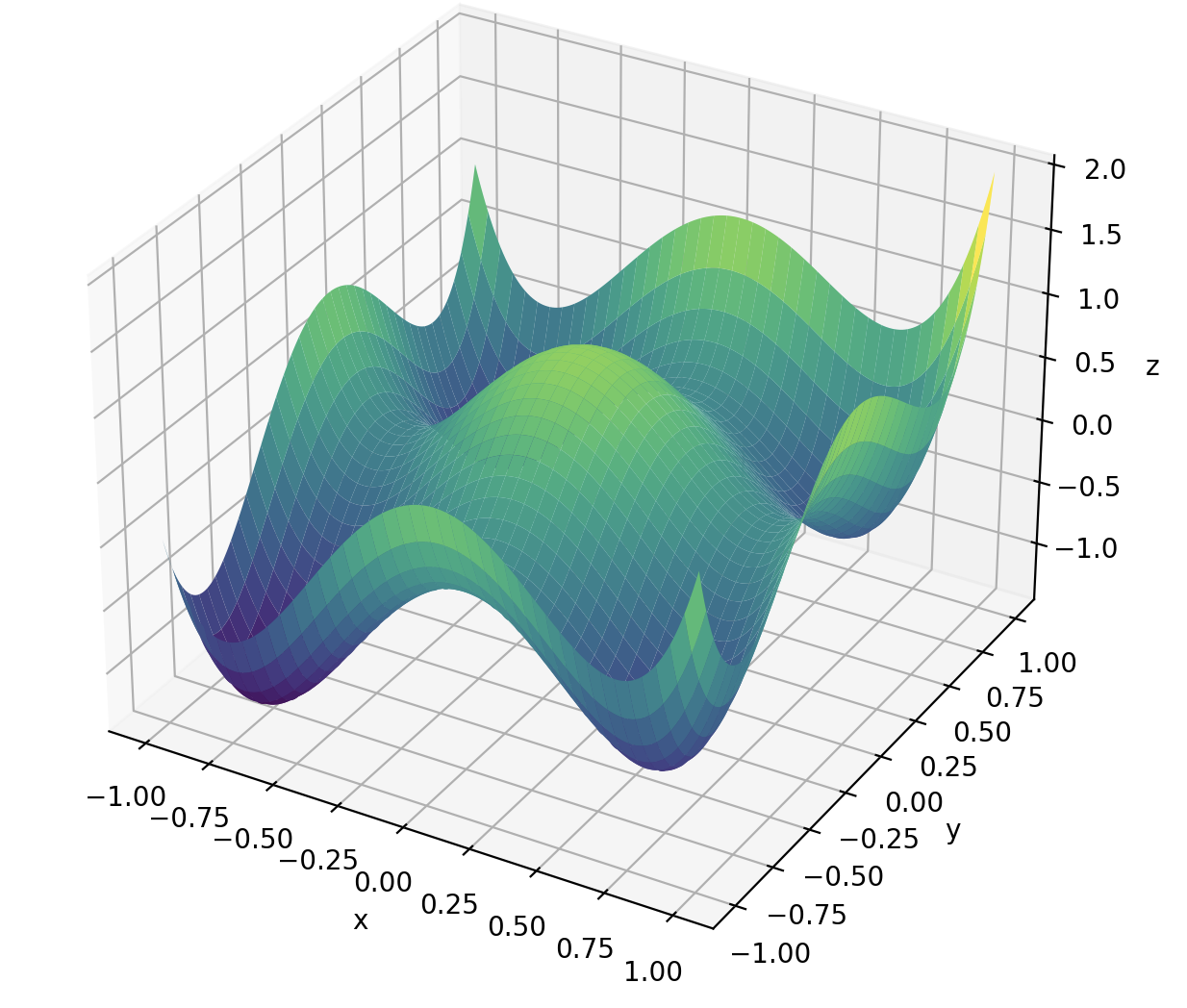}
  \caption{\small 3-D View}
  \label{fig:sub2.2}
\end{subfigure}
\caption{\footnotesize Two-dimensional (2-D) rendering of the  non-convex function $g_2(x)$ from Example~\ref{sec:numerical_example_2}. For an exponential number of local minima, as well as the asymmetry in this function, global optimization is intractable using traditional techniques.}
\label{fig:dense}
\end{figure}
As with Example~\ref{sec:numerical_example_1}, the functions of Example~\ref{sec:numerical_example_2} are also highly non-convex and have exponentially many local minima ($2^D$ in dimension $D$). Every function possesses a unique minimum with value approximately equal to $-1.3911$ located near $(-0.75553, ..., -0.75553)$; i.e.,
\begin{equation}
    g^*_D := g_D(-0.75553,..., -0.75553) \approx -1.3911.
\end{equation}

\begin{table}[htbp!]
    \centering
    \small
    \caption{\footnotesize Parameters used for performing numerical experiments related to Numerical Example~\ref{sec:numerical_example_2}}
    \begin{tabular}{|c|c|c|} \hline
        \multirow{4}{*}{Formulation} & Size of moment matrices ($(d+1) \times (d+1)$) & $5$  \\ 
            & Penalty parameter $(\gamma)$  &  $10.0$ \\
            & Burer-Monteiro rank & $5$ (full) \\
            & Number product measures ($L$) & $6$ \\ \hline
        \multirow{3}{*}{Solver (KKT)} & Stopping criterion (relative gradient) & $10^{-2}$ \\
            & Stopping criterion (absolute value) & $10^{-3}$ \\
            & Stopping criterion (absolute feasibility) & $10^{-2}$ \\ \hline
        \multirow{4}{*}{Solver (Primal/L-BFGS)} & Approximation order  &  $40$ \\
                & Wolfe Beta factor ($\beta$) & $0.4$ \\
                & Stopping criterion (relative gradient) & $ 10^{-2} $ \\ 
                & Stopping criterion (absolute value) & $10^{-5}$ \\\hline
    \end{tabular}
    \label{tab:param_family_2}
\end{table}

Figure~\ref{fig:family_2_error} shows how our method  accurately (i.e., within $10^{-2}$) computes the relative error in the optimal value and location for Example~\ref{sec:numerical_example_2}.  Figure~\ref{fig:family_2_timing} exhibits a computation time scaling of $O(D^4)$. As the driving computational component (the objective evaluation) carries a cost of $O(D^3)$ (resulting in a total cost of $O(D^4)$ according to Table~\ref{tab:cost}), the results are in line with expectations.

\begin{figure}[htbp!]
    \centering
    \includegraphics[width=4in]{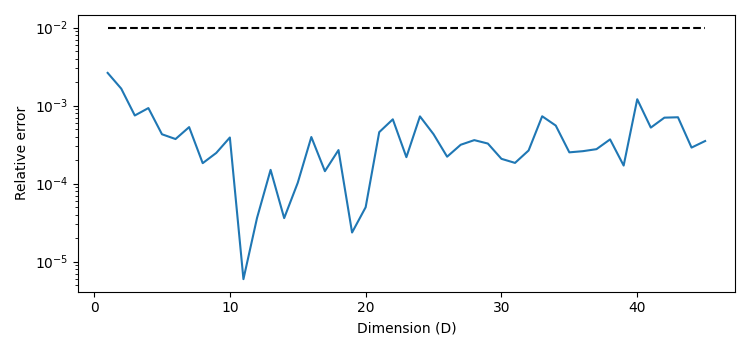}
    \includegraphics[width=4in]{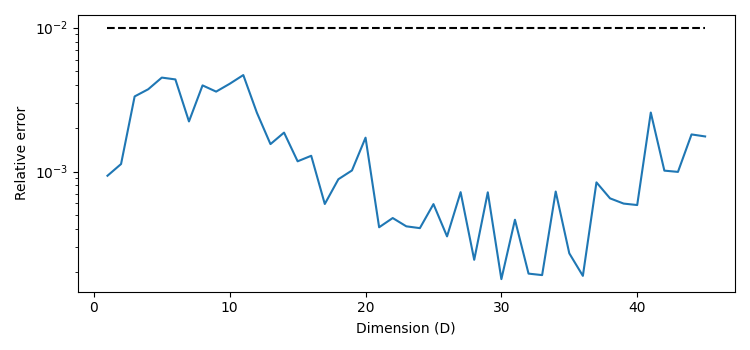}
    \caption{\footnotesize Relative error vs problem dimension ($D$) for dimension $1$ to $45$ for Example~\ref{sec:numerical_example_2}. {\emph Top}: optimal value; {\emph Bottom}: optimal location. The error lies below the expected upper bound (dotted black line) of $10^{-2}$ in every case.}
    \label{fig:family_2_error}
\end{figure}
\begin{figure}[htbp!]
    \centering
    \includegraphics[width=4in]{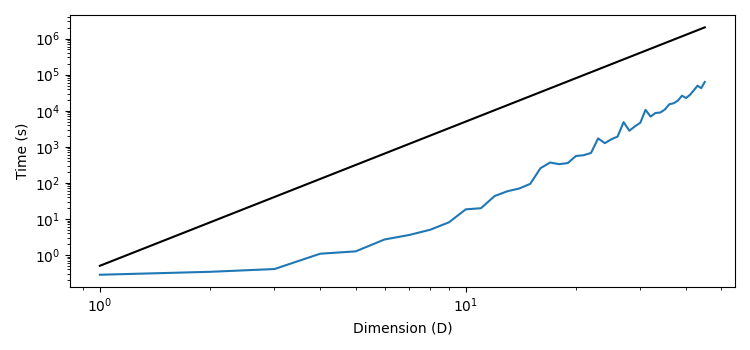}
    \caption{\footnotesize Computation (wall) time (s) vs problem dimension ($D$) for Example~\ref{sec:numerical_example_2}. Our implementation exhibits a polynomial (quartic, degree-$4$; black line) scaling on problems previously considered intractable.}
    \label{fig:family_2_timing}
\end{figure}

\section{Conclusion}
\label{sec:conclusion}
In this paper, we have introduced a new framework for efficiently reformulating general polynomial optimization problems over the hypercube $[-1,1]^D$ as non-linear problems with no spurious local minima. We have rigorously demonstrated the correctness of our approach, and have further demonstrated the correctness, performance, and practicality of our implementation using polynomial optimization problems, which are intractable by standard techniques.

From a theoretical standpoint, future work will target the treatment of more complex polynomial (e.g., semi-algebraic) constraints. Performance improvements, both algorithmic and from a software optimization perspective (including parallelism), will also take a prominent role in the near future.

% \section{Data Availability Statement}
% The data produced for the redaction of this paper will be made available under reasonable requests satisfying the policy of Qualcomm Technologies Inc. for the public release of internal research material.

%\newpage
\appendix

\section{Proofs}
\label{sec:proofs}

\begin{theorem}[Powers, Reznick, from Fekete~\cite{powers2000polynomials}] 
\label{thm:fekete}
Let $p \in \mathbb{P}_{1,d}$ be a 1D polynomial of degree $d$ non-negative on $[-1,1]$. Then,
\begin{equation}
	p(x) = (f(x))^2 + (1-x^2) \, (g(x))^2
\end{equation}
for some polynomials $f(x)$ and $g(x)$ of degree at most $d$ and $d-1$ respectively.
\end{theorem}

\begin{theorem}[ \bf Riesz extension theorem,~\cite{rudin1986real}] 
\label{thm:rieszextension}
Let $E$ be a real vector space, $F\subset E$ be a subspace and $K \subset E$ a convex cone. Assume $E = K+F$ and let $\phi : F \rightarrow \mathbb{R}$ be a linear functional over $F$ that is non-negative over $K \cap F$; i.e., $\phi(y) \geq 0 \;\; \forall \;\; y \in K \cap F$.  Then, there exists a linear functional
 $$\psi : E \rightarrow \mathbb{R},$$
such that
\begin{align}
	\psi(x) &= \phi(x) \;\; \forall \;\; x \in F, \quad{ and} \\
	\psi(y) &\geq 0 \;\; \forall \;\; y \in K.
\end{align}	
\end{theorem}

\begin{theorem}[\bf Riesz-Markov,\cite{reed1980functional}]\footnote{This statement is adapted from~\cite{reed1980functional}, Theorem IV.14, and the preceding remarks on the one-to-one correspondence of Baire measures and regular Borel measures.}
\label{thm:rieszmarkov}
Let $X$ be a compact Hausdorff space and assume that $\psi : C (X) \rightarrow \mathbb{R}$  is a continuous linear functional over $C(X)$, the space of functions continuous on $X$. Assume further that $\psi(\cdot)$ is non-negative over $C(X)$; i.e., 
\begin{equation}
	\psi(f) \geq 0 \;\;  \forall \;\; f \in C(X) \, : \, f(x) \geq 0 \;\;\forall \;\; x \in X.
\end{equation}		
Then, there exists a unique regular Borel measure $\mu (\cdot)$, such that
\begin{equation}
	\psi(f) = \int_X f(x) \, \mathrm{d} \mu(x).
\end{equation}
\end{theorem}

\begin{theorem}[\bf B.L.T., \cite{reed1980functional}] 
\label{thm:BLT}
Suppose $T:V_1 \rightarrow V_2$ is a bounded linear transformation from a normed linear space $\langle V_1, \lvert \lvert \cdot \rvert \rvert_1\rangle$ to a complete normed linear space $\langle V_2, \lvert \lvert \cdot \rvert \rvert_2\rangle$. Then, $T$ can be uniquely extended to a bounded linear transformation (with the same bound), $\tilde{T}$, from the completion of $V_1$ to $\langle V_2, ||\cdot||_2\rangle$. 
\end{theorem}

\begin{proposition}
[\bf Stone-Weierstrass (corollary), \cite{reed1980functional}] 
\label{thm:stoneweierstrass}
The polynomials are dense in $C[a,b]$, the space of real-valued functions continuous over the interval $[a,b]$, for any finite real number $a < b$.
\end{proposition}

Proposition~\ref{prop:measureextension} from Section~\ref{sec:reformulation_theory} is a known result and may be found in, e.g., \cite{nie2023moment} and \cite{curto1991recursiveness}. We restate the proposition here and repeat the proof in this context for the sake of completeness.

\measureextension*
\begin{proof}
Let $F := \mathbb{P}_{1,d} $ be the vector space of 1D polynomials of degree at most $d$, $E = \mathbb{P}_{1,\infty}:= \mathbb{R}[x]$, the space of 1D polynomials with real coefficients of any finite degree, and
\begin{equation}
	K := \left \{ p \in  \mathbb{P}_{1,\infty} \, : \, p(x) \geq 0 \;\; \forall \;\; x \in [-1,1]  \right \}
\end{equation}	
be the cone of 1D polynomials of any finite degree non-negative on $[-1,1]$. We first demonstrate that
$$ E = K+F.$$
Given $p \in E$, let $h(x) \equiv \min_{x \in [-1,1]} p(x) > -\infty$. This constant polynomial is well-defined since continuous functions reach their extrema over compact sets. Further, $h \in F$ as $F$ contains the constant polynomials by definition. Finally, the function
\begin{equation}
	p(x) - h(x) = p(x) - \min_{x \in [-1,1]} p(x) \geq 0 \;\; \forall \;\; x \in [-1,1]
\end{equation}
implies that $p-h \in K$ by construction. Since this holds for every element $p\in E$, this demonstrates the claim.

Now, consider the following family of linear functionals over $F$: $\forall i=1,..., D$, where $q = \sum_{k=0}^d q_k \,x^k \in F$.  Define
\begin{equation}
	\phi_i (q(x_i)) :=  \sum_{{\bf n}_i=0}^d   q_{{\bf n}_i}  \, \mu_{i,{\bf n}_i}.
\end{equation}
Note in particular that $\phi_i (x_i^{{\bf n}_i}) = \mu_{i,{\bf n}_i}$. We claim that each functional is non-negative on $F \cap K$. Indeed, assume $q$ is non-negative; i.e., $q \in F \cap K$. Then, by Theorem \ref{thm:fekete} there exists polynomials $f \in \mathbb{P}_{D,d}$ and $g \in \mathbb{P}_{D,d-1}$, such that
\begin{align}
	\phi_i (q(x_i)) &= \phi_i \left (f^2(x_i) + (1-x_i^2) \, g^2(x_i) \right )\\
     &= \sum_{{\bf m}_i=0, {\bf n}_i=0}^d f_{{\bf m}_i} f_{{\bf n}_i} \, \phi_i(x^{{\bf m}_i+{\bf n}_i}) + \sum_{{\bf m}_i=0, {\bf n}_i=0}^{d-1} g_{{\bf m}_i} g_{{\bf n}_i} \, \phi_i( (1-x_i^2) \, x^{{\bf m}_i+{\bf n}_i}) \\
     &= \sum_{{\bf m}_i=0, {\bf n}_i=0}^d f_{{\bf m}_i} f_{{\bf n}_i} \, \mu_{i,{\bf m}_i+{\bf n}_i} + \sum_{{\bf m}_i=0, {\bf n}_i=0}^{d-1} g_{{\bf m}_i} g_{{\bf n}_i} \, \left  (\mu_{i,{\bf m}_i+{\bf n}_i} - \mu_{i,{\bf m}_i+{\bf n}_i+2} \right ) \\
	 &= f^T \, \mathcal{M}_d ( \mu_i ) \, f + g^T \,  \mathcal{M}_{d-1} ( \mu_i;(1-x_i^2) )   \, g \\
	 & \geq 0,
\end{align}
where the last inequality follows by assumption since $\mathcal{M}_d ( \mu_i ) \succeq 0$ and $\mathcal{M}_{d-1} ( \mu_i; \, 1-x_i^2 ) \succeq 0 $ for every $i$. Therefore, $\phi_i(\cdot)$ are non-negative functionals on $F \cap K$ and thus satisfies the hypothesis of the Riesz extension theorem (Theorem \ref{thm:rieszextension}). This implies that for each $\phi_i(\cdot)$ there exists an extension $\psi_i : E \rightarrow \mathbb{R} $ that is non-negative on $K$. 

Recall that $E$ corresponds to $\mathbb{P}_{1,\infty}$, which is a dense subspace of the space $C[-1,1]$ of functions continuous on $[-1,1]$ by Theorem \ref{thm:stoneweierstrass} (Stone-Weierstrass). For each $i$, we can therefore invoke the B.L.T. theorem (Theorem \ref{thm:BLT}) with $\psi_i(\cdot)$ a linear transformation from the normed space $\langle \mathbb{P}_{1,\infty}, \lvert \lvert \cdot \rvert \rvert_\infty \rangle$ to the complete normed space $\langle \mathbb{R}, \lvert \cdot \rvert \rangle$. In this way, we can assume w.l.o.g. that
$$\psi_i : C([-1,1]) \rightarrow \mathbb{R}$$
is a continuous linear functional defined over the entire space $C([-1,1])$ and non-negative over the cone of continuous functions non-negative on $[-1,1]$.   Under these conditions, the Riesz-Markov theorem (Theorem \ref{thm:rieszmarkov}) implies the existence of unique regular Borel measures $\{\mu_i (\cdot)\}$ supported on $[-1,1]$, such that
\begin{equation}
	 \mu_{i,{\bf n}_i} = \phi_i (x_i^{{\bf n}_i}) = \int_{[-1,1]} x_i^{{\bf n}_i} \, \mathrm{d} \mu_i (x_i),
\end{equation} 
where $\mu_{i,0}=1$. Finally, it follows as a consequence of Caratheodory theorem (see~\cite{cohn2013measure}, Theorem 5.1.3) that there exists a unique regular product measure
$$\mu(\cdot) = \prod_{i=1}^D \mu_i(\cdot)$$
with moments 
\begin{equation}
	 \mu_{\bf n}  =  \int_{[-1,1]^D} x^{\bf n} \, \mathrm{d} \mu(x)   =  \prod_{i=1}^D \, \int_{[-1,1]} x_i^{{\bf n}_i} \, \mathrm{d} \mu_i (x_i) = \prod_{i=1}^D \mu_{i,{\bf n}_i}, \; 0 \leq {\bf n}_i \leq d,
\end{equation} 
and the constraints $\mu_{i,0} = 1, \forall \, i$ further imply that
\begin{equation}
    \mu_{\bf n}([-1,1]^D) = \prod_{i=1}^D \mu_{i}([-1,1]) = \prod_{i=1}^D \mu_{i,0} = 1,
\end{equation}
which proves the desired result. 
\end{proof}

We now introduce some quantities that will prove important for the remainder of the proof. For $0 < L \in \mathbb{N}$, define 
$\mathcal{F}_L$ and $\mathcal{S}_L$ to be
\begin{align}
\begin{split}
     \mathcal{F}_L :=
     \bigg\{   \bigg\{\mu_{\bf n}:= {}& \sum_{l=1}^L \prod_{i=1}^D \mu_{i, {\bf n}_i}^{(l)}  \bigg\}_{\lvert {\bf n} \rvert \leq d}  \, : \, \{\mu_i^{(l)} \}_{i,l=1}^{D,L} \in \mathbb{R}^{(2d+1) \times D \times L},\, \\   
     & \, \mathcal{M}_d(\mu^{(l)}_i), \, \mathcal{M}_{d-1}(\mu^{(l)}_i ; (1-x_i^2)) \succeq 0; \forall \, i,l, \, \mu_{0} = 1 \bigg\},
\end{split}
\end{align}
\begin{align}
     \mathcal{S}_L &:= \left \{ \mu(\cdot) := \sum_{l=1}^L \prod_{i=1}^D  \mu_i^{(l)} (\cdot)  \, :  \,  \mu^{(l)}_i (\cdot) \in \mathbb{B}_1; \forall \; i,l, \, \mu\left( [-1,1]^D \right ) = 1\right \},
\end{align}
where, as before, the symbol $\mathbb{B}_D$ represents the set of finite regular Borel measures supported over $[-1,1]^D$. Also note that $\mathcal{F}_L$ corresponds to the image of feasible set under the map $\phi(\cdot)$ in Problem~\ref{eq:reformulationcontinuous}. With this notation, we may now proceed.

\begin{proposition}
    \label{prop:eqsets}
There exists a continuous surjective map,
$$\zeta :\mathcal{S}_L \rightarrow \mathcal{F}_L,$$
between $ \mathcal{S}_L$ and $ \mathcal{F}_L$ that satisfies
\begin{equation}
    \left [  \zeta \left ( \mu(\cdot) \right ) \right ]_{\bf n} =  \sum_{l=1}^L \prod_{i=1}^D \mu_{i, {\bf n}_i}^{(l)}
\end{equation}
for every $  \lvert n \rvert  \leq d$.
\end{proposition}
\begin{proof}
    Given an element
    \begin{equation}
         \mu(\cdot) := \sum_{l=1}^L \prod_{i=1}^D  \mu_i^{(l)} (\cdot)  \in \mathcal{S}_L,
    \end{equation}
    we define the map $\zeta(\cdot)$ through
    \begin{equation}
         \left [  \zeta \left ( \mu(\cdot) \right ) \right ]_{\bf n} := \int_{[-1,1]^D}  \, x^{{\bf n}} \, \mathrm{d} \mu (x) = \sum_{l=1}^L \prod_{i=1}^D  \int_{[-1,1]^D}  \, x_i^{{\bf n}_i} \, \mathrm{d} \mu_i^{(l)} (x_i) = \sum_{l=1}^L \prod_{i=1}^D \mu_{i, {\bf n}_i}^{(l)} .
    \end{equation}
   First, we demonstrate that $\zeta \left ( \mu(\cdot) \right )$ belongs to $\mathcal{F}_L$. To do so, note that
    \begin{equation}
         \mu_0  =   \sum_{l=1}^L \prod_{i=1}^D  \mu^{(l)}_{i,0} = \left [  \zeta \left ( \mu(\cdot) \right ) \right ]_0 = \int \, \mathrm{d} \left ( \sum_{l=1}^L \prod_{i=1}^D     \mu_i^{(l)} (x_i) \right ) = \mu\left ( [-1,1]^D \right )  = 1.
    \end{equation}
    Thus, it suffices to shows that $\mathcal{M}_d(\mu^{(l)}_i)$ and $\mathcal{M}_{d-1}(\mu^{(l)}_i ;\, (1-x_i^2))$ are positive semi-definite for each $i=1,...,D$ and $l=1,...,L$. To observe this, note that following our definition,
    \begin{align}
        f^T \, \mathcal{M}_d(\mu^{(l)}_i) \, f &= \int_{-1}^1 \, f^2(x_i) \, \mathrm{d}\mu^{(l)}_i(x_i) \geq 0,\\
         g^T \, \mathcal{M}_{d-1}(\mu^{(l)}_i; (1-x_i^2)) \, g &= \int_{-1}^1 \, (1-x_i^2) \, g^2(x_i) \, \mathrm{d}\mu^{(l)}_i(x_i) \geq 0,
    \end{align}
    for every $f\in\mathbb{P}_{1,d} $ and $g\in\mathbb{P}_{1,d-1} $, positivity follows from the fact the integrands are non-negative over $[-1,1]$ and the measures are positive and supported over $[-1,1]$. Since this holds for every vector of coefficients $f$ and $g$, we conclude that all matrices are indeed positive semi-definite, thus proving the claim.
    
    Regarding surjectivity, consider an element $\mu \in \mathcal{F}_L$.
    Following Proposition~\ref{prop:measureextension}, for each vector of moments $\mu^{(l)}_{i}$, there exists of a regular Borel measure $\mu^{(l)}_i(\cdot)$ supported over $[-1,1]$, such that
    \begin{equation}
        \int_{[-1,1]^D} \, x_i^{{\bf n}_i} \, \mathrm{d} \mu_i^{(l)} (x_i) = \mu^{(l)}_{i,{\bf n}_i}
    \end{equation}
    for every $0 \leq {\bf n}_i   \leq 2d+1$ that also satisfies: $\sum_{l=1}^L \prod_{i=1}^D \mu^{(l)}_{i,0} = 1 $. The element
    \begin{equation}
        \mu(\cdot) := \sum_{l=1}^L \prod_{i=1}^D  \mu_i^{(l)} (\cdot) 
    \end{equation}
    therefore belongs to $\mathcal{S}_L$ and maps to $\mu$ through $\zeta(\cdot)$ via substitution. Since this is true for every element of $\mathcal{F}_L$, this demonstrates surjectivity. 
    
    Finally, given $\lvert x^{\bf n} \rvert \leq 1$,  $\forall x \in [-1,1]^D$, where $\lvert  \cdot \rvert$ indicates total variation (see, e.g., \cite{cohn2013measure}), continuity follows from the fact that for any two element $\mu(\cdot), \nu(\cdot) \in \mathcal{S}_L$, 
    \begin{align}
       &\left \lvert   \left [  \zeta \left ( \sum_{l=1}^L \prod_{i=1}^D  \mu_i^{(l)} (\cdot)   \right ) \right ]_{\bf n}  -    \left [  \zeta \left ( \sum_{l=1}^L \prod_{i=1}^D  \nu_i^{(l)} (\cdot)   \right ) \right ]_{\bf n}  \right \rvert  \\
       &= \left \lvert \int_{[-1,1]^D} \, x^{\bf n}  \, \mathrm{d}  \left ( \sum_{l=1}^L \prod_{i=1}^D  \mu_i^{(l)} (\cdot) - \sum_{l=1}^L \prod_{i=1}^D  \nu_i^{(l)} (\cdot) \right )  \right \rvert \\
       &\leq \left  \lvert  \sum_{l=1}^L \prod_{i=1}^D  \mu_i^{(l)} (\cdot) - \sum_{l=1}^L \prod_{i=1}^D  \nu_i^{(l)} (\cdot) \right \rvert
    \end{align}
    by the triangle inequality. 
\end{proof}

\begin{corollary}
    \label{cor:equivproblem}
    Consider the problems
    \begin{equation}
    \label{eq:problemmoments}
        \min_{ \mu \in \mathcal{F}_L }   \sum_{ n \in \mathrm{supp}(p)} \, p_n \left ( \sum_{l=1}^L  \prod_{i=1}^D \mu^{(l)}_{i,{\bf n}_i} \right ),
    \end{equation}
    and
    \begin{equation}
        \label{eq:problemmeasure}
        \min_{ \mu(\cdot) \in \mathcal{S}_L } \; \int p(x) \, \mathrm{d} \left (  \sum_{l=1}^L  \prod_{i=1}^D \mu^{(l)}_i (x_i) \right ).  
    \end{equation}
    The following holds:
    \begin{enumerate}
        \item Both problem share the same global minimum value. 
        \item $ \mu(\cdot) \in \mathcal{S}_L$ corresponds to a global minimum of Problem \ref{eq:problemmeasure} if and only if $ \zeta(\mu(\cdot)) \in \mathcal{F}_L$ corresponds to a global minimum of Problem \ref{eq:problemmoments}.
        \item If $ \mu \in \mathcal{F}_L$ corresponds to a local minimum of Problem \ref{eq:problemmoments}, then any element of the non-empty set $\zeta^{-1}(\mu) \in \mathcal{S}_L$ consists of local minima of Problem \ref{eq:problemmeasure}.
    \end{enumerate}
\end{corollary}

\begin{proof} 
First, note that by definition, the objective functional satisfies
\begin{align}
    \label{eq:invfunc}
    \int p(x) \, \mathrm{d} \left(\sum_{l=1}^L  \prod_{i=1}^D \mu^{(l)}_i(x_i) \right ) &=  \sum_{{\bf n} \in \mathrm{supp}(p)} \, p_{\bf n} \left (  \left [  \zeta \left(\sum_{l=1}^L  \prod_{i=1}^D \mu^{(l)}_i(\cdot) \right )  \right ]_{{\bf n}} \right ).
     % \sum_{n \in \mathrm{supp}(p)} \, p_n \left ( \sum_{l=1}^L  \prod_{i=1}^D \mu^{(l)}_{i,{\bf n}_i} \right ) \\
\end{align}
This implies that $ \mu^*(\cdot) \in \mathcal{S}_L$ is a global minimum of Problem \ref{eq:problemmeasure} if and only if $\mu^* := \zeta(\mu^*(\cdot))$ is a global minimum of Problem \ref{eq:problemmoments}. Indeed, if a point in $\mathcal{F}_L$ achieves a strictly lower value, then any point belonging to the pre-image of that point under $\zeta(\cdot)$~\footnote{This pre-image is non-empty by surjectivity. See Proposition \ref{prop:eqsets}.} achieves a lower value and vice-versa. This is a contradiction and demonstrates Points $1$ and $2$.

For the third claim, consider a local minimum $\mu \in \mathcal{F}_L $. Let $\mu(\cdot)\in \mathcal{S}_L $ be any element of the pre-image of this local minimum under $\zeta(\cdot)$; i.e., $\mu(\cdot) \in    \zeta^{-1} \left (  \mu \right )$, which is non-empty due to surjectivity. Since $\mu$ is a local minimum, there exists an open neighborhood $\mathcal{N}$ of $\mu$ in $\mathcal{F}_L$ where the functional may only achieve values that are either greater than or equal to the current value. Now, consider the set
\begin{equation}
     \zeta^{-1} (\mathcal{N}) \subset \mathcal{S}_L,
\end{equation}
which is open and contains $\zeta^{-1} \left (  \mu \right )$ since $\zeta(\cdot)$ is continuous and surjective. From the invariance of the functional value (Equation~\eqref{eq:invfunc}), it follows that every element of this set must lead to a value in Problem \ref{eq:problemmeasure} that is either greater than or equal to the current value; i.e., the set $\zeta^{-1} (\mu ) $  must consist of local mimina of Problem \ref{eq:problemmeasure}. This proves Point $3$.
\end{proof}

% An interesting observation is that problem \ref{eq:problemmoments} is, in fact, \emph{convex}. It is the parametrization of the feasible set through the map $\phi(\cdot)$ that ultimately introduce non-convexity.

\equivalence*
\begin{proof}
We claim that the optimal value $\pi^*$ of Problem \ref{eq:reformulationcontinuous} equals

\begin{equation}
\min_{x \in [-1,1]^D} \, p(x). 
\end{equation}
To see why this claim is true, begin by letting $x^* \in [-1,1]^D$ be an optimal location~\footnote{This optimal location exists since $p(x)$ is continuous and $[-1,1]^D$ is compact.} and note that the choice,
\begin{equation}
   \mu^{(0)}_{i,{\bf n}_i } = [\delta^{(0)}_{x^*_i}]_{{\bf n}_i} =  (x_i^*)^{{\bf n}_i} 
\end{equation}
and $\mu^{(1)}_{i, {\bf n}_i} \equiv 0  \; \forall \, {\bf n}$, is feasible for Problem \ref{eq:reformulationcontinuous}.
% where $\delta(\cdot)$ is a point (delta) measure centered at the origin, belongs to $\mathcal{S}_L$. 
Therefore, 
\begin{equation}
    \pi^* \leq \sum_{{\bf n} \in \mathrm{supp}(p)} p_{\bf n} \, \prod_{i=1}^D \mu^{(0)}_i  = \sum_{{\bf n} \in \mathrm{supp}(p)} p_{\bf n} \, (x^*)^{\bf n} = p(x^*).
\end{equation}
For every element of $\mathcal{S}_L$, consider that
\begin{align}
    \label{eq:equivmin_upperbound}
    \int_{[-1,1]^D} &p(x) \, \mathrm{d} \left(\sum_{l=1}^L \prod_{i=1}^D \mu_i(x_i) \right ) \geq  p(x^*)   \, \int_{[-1,1]^D} \mathrm{d} \left ( \sum_{l=1}^L \prod_{i=1}^D \mu^{(l)}_{i}(x_i) \right )  = p(x^*).
\end{align}
It then follows by Proposition \ref{prop:measureextension} that
\begin{equation}
    \pi^* \geq \min_{x \in [-1,1]^D} \, p(x).
\end{equation}
Together with Equation~\eqref{eq:equivmin_upperbound}, this  implies that
\begin{equation}
    \pi^* = \min_{x \in [-1,1]^D} \, p(x)
\end{equation}
as claimed.
\end{proof}

\begin{proposition}
    \label{prop:descentmeasure}
    Consider an element of $\mathcal{S}_2$ of the form,
        \begin{equation}
           \mu(\cdot) := \mu^{(0)}(\cdot)  +  \mu^{(1)}(\cdot ) :=  \prod_{i=1}^D \mu_i^{(0)}(\cdot)  +  \prod_{i=1}^D \mu_i^{(1)}(\cdot).
        \end{equation}
    Assume such element is not a global minimum of the problem
    \begin{equation}
            \min_{ \mu(\cdot) \in \mathcal{S}_2}\; \int p(x) \, \mathrm{d} \left( \prod_{i=1}^D \mu_i^{(0)}(x_i)  + \prod_{i=1}^D \mu_i^{(1)}(x_i) \right ).
    \end{equation}
    Then, if either ($1$) $\mu^{(0)}\left (  [-1,1]^D \right )  \in \{0,1\}$, or ($2$) $\mu^{(0)}\left (  [-1,1]^D \right ) \in (0,1)$, and
    \begin{equation}
       \frac{1}{\mu^{(0)}([-1,1]^D) } \, \int p(x) \, \mathrm{d}  \mu^{(0)}(x) \not =  \frac{1}{\mu^{(1)}([-1,1]^D) } \, \int p(x) \, \mathrm{d}  \mu^{(1)}(x),
    \end{equation}
    then there exists a path $\pi(t)$, $t \in [0,1]$ lying entirely in $\mathcal{S}_2$ and originating from this point along which $\int p(x) \, \mathrm{d} \pi(t)(x)$ monotonically decreases.
\end{proposition}
\begin{proof}
Consider a sub-optimal element $\mu(\cdot)$ of $\mathcal{S}_2$ of the form,
\begin{equation}
    \mu(\cdot) =  \mu^{(0)}(\cdot) +  \mu^{(1)}(\cdot) = \prod_{i=1}^D \mu^{(0)}_i(\cdot) +   \prod_{i=1}^D \mu^{(1)}_i(\cdot).
\end{equation}
From Theorem \ref{thm:equivalence}, we know that the optimal value is given by
\begin{equation}
     \min_{x \in [-1,1]^D} p(x) = p(x^*)
\end{equation}
for some $x^* \in [-1,1]^D$. Therefore, sub-optimality implies that
\begin{equation}
    \int p(x) \, \mathrm{d} \mu(x) > p(x^*)
\end{equation}
We consider two cases: 
\begin{enumerate}
    \item The case where $\mu^{(0)}\left (  [-1,1]^D \right ) \in \{0,1\}$;
    \item The case where $\mu^{(0)}\left (  [-1,1]^D \right ) \in (0,1)$.
\end{enumerate}

\noindent
{\bf Case 1:} $\mu^{(0)}\left (  [-1,1]^D \right ) \in \{0,1\}$\\
Assume w.l.o.g. that $\mu^{(0)}\left (  [-1,1]^D \right ) = 1$, so that $\mu(\cdot) = \prod_{i=1}^D \mu^{(0)}_i(\cdot)$.  Consider the measure,
\begin{equation}
    \delta_{x^*}(\cdot) = \prod_{i=1}^D \delta(\cdot - x_i^*),
\end{equation}
where $\delta(\cdot)$ is a point (delta) measure supported at the origin. For $t \in [0,1]$, consider the parametrized path,
\begin{equation}
    \pi(t) := \prod_{i=1}^D \pi^{(0)}_{i} (t) + \prod_{i=1}^D \pi^{(1)}_{i} (t),
\end{equation}
where
\begin{align}
	\pi^{(0)}_{i} (t) &:= (1-t)^{1/D}  \,  \mu^{(0)}_{i}(\cdot),\quad{and}\\
 	\pi^{(1)}_{i} (t) &:= t^{1/D}\, \delta (\cdot - x_i^*). \\
\end{align}
Note that this path starts at $ \pi(0) = \mu^{(0)}(\cdot) = \mu(\cdot)$, we claim that (1) it lies entirely in $\mathcal{S}_2$, and (2) the objective monotonically decreases along it. This can be seen by noting that $\delta(\cdot - x_i^*) \in \mathbb{B}_1, \forall i$ by construction, and $\mu_i^{(0)}(\cdot) \in \mathbb{B}_1, \forall i$ by assumption. By substitution, it follows that for all $t \in [0,1]$
\begin{align}
    \pi(t) =   (1-t) \, \prod_{i=1}^D \mu_i^{(0)}(\cdot) + t \prod_{i=1}^D \delta(\cdot - x_i^*)   = (1-t) \, \mu^{(0)}_i (\cdot) + t \, \delta_{x^*} (\cdot), 
\end{align}
and
\begin{align}
	& \pi(t)([-1,1]^D) = (1-t) \, \mu^{(0)} \left ( [-1,1]^D\right ) + t \, \delta_{x^*} \left ( [-1,1]^D \right ) =  (1-t) + t = 1.
\end{align}
This implies that $\pi(t)$ belongs to $\mathcal{S}_2$ per claim (1). For claim (2), consider the functional value along the path; namely,
\begin{align}
  \int p(x) \; \mathrm{d}\pi(t)(x) &=  (1-t)\,  \int p(x) \, \mathrm{d} \mu^{(0)}(x) + t\,  \int p(x)\, \delta_{x^*}(x) \, \mathrm{d}x \\
 &= \int p(x) \, \mathrm{d} \mu^{(0)}(x) + t\, \left ( \int p(x) \, \delta_{x^*}(x) \, \mathrm{d}x  - \int p(x) \, \mathrm{d} \mu^{(0)}(x) \right ).
\end{align}
This, coupled with the fact that $\int p(x) \, \mathrm{d} \mu^{(0)}(x) > p(x^*)  $ by assumption of sub-optimality,  implies 
\begin{equation}
    \frac{d}{d\, t} \, \int p(x) \; \mathrm{d}\pi(t)(x) =  p(x^*) -  \int p(x) \, \mathrm{d} \mu^{(0)}(x) < 0, 
\end{equation}
demonstrating monotonic decrease of the functional along the path as claimed.

\noindent
{\bf Case 2:} $\mu^{(0)}\left (  [-1,1]^D \right ) \in (0,1)$. 
Begin by letting
\begin{equation}
    \alpha := \mu^{(0)}\left (  [-1,1]^D \right ).
\end{equation}
Under our hypothesis, it can be assumed w.l.o.g. that
\begin{equation}
    \label{eq:descent_case2}
  \frac{1}{1-\alpha} \, \int p(x) \, \mathrm{d} \mu^{(1)}(x) < \frac{1}{\alpha} \, \int p(x) \, \mathrm{d} \mu^{(0)}(x) .
\end{equation}
Again, consider a parametrized path $\pi(t)$, $t\in[0,1]$ such that
\begin{align}
	   \pi(t) :=   \prod_{i=1}^D \pi^{(0)}_{i} (t) +  \prod_{i=1}^D \pi^{(1)}_{i} (t),  
\end{align}
where
\begin{align}
	\pi^{(0)}_{i} (t) & =  \left ( 1 - t \right )^{1/D} \, \mu_{i}^{(0)}(\cdot), \quad{and} \\
	\pi^{(1)}_{i} (t) &= \left( 1 + t \, \frac{\alpha}{1-\alpha} \right )^{1/D} \, \mu_{i}^{(1)} (\cdot).
\end{align}
Beginning at $ \pi(0) = \mu^{(0)}(\cdot) +  \mu^{(1)}(\cdot) =  \mu(\cdot)$, we claim that (1) this path lies entirely in $\mathcal{S}_2$, and (2) the objective monotically decreases along this path. Indeed, $\forall i,  \mu_{i}^{(0)}(\cdot), \, \mu_{i}^{(1)}(\cdot) \in \mathbb{B}_1$ by assumption. Additionally, 
\begin{align}
    \pi(t) &=   (1-t) \, \prod_{i=1}^D \mu^{(0)}_i (\cdot) +   \left (1 + t \, \frac{\alpha}{1-\alpha} \right )  \, \prod_{i=1}^D \mu_i^{(1)} (\cdot)\\
    &= (1-t) \,  \mu^{(0)} (\cdot) +   \left (1 + t \, \frac{\alpha}{1-\alpha} \right )  \,  \mu^{(1)} (\cdot) ,
\end{align}
which implies that $\forall t \in [0,1]$
 \begin{align}
	\pi(t)([-1,1]^D) &=  (1-t) \, \alpha  +  \left (1 + t\, \frac{\alpha}{1-\alpha} \right ) \, (1-\alpha)\\
	 &= \alpha +  (1-\alpha)  - t\, \alpha + t \alpha \\
    &= 1.
 \end{align}
This demonstrates that $\pi(t)$ belongs to $\mathcal{S}_2$ per claim (1). Finally, the functional value along the path is
\begin{align}
	  \int p(x) \, \mathrm{d} \pi (x)  &=  \int p(x) \, \mathrm{d} \left ( (1-t) \, \mu^{(0)}(x) +  \left(1+    t \,  \frac{\alpha}{1-\alpha} \right)\,  \mu^{(1)}(x) \right )\\
	  &=   \int p(x) \, \mathrm{d}  \mu^{(0)} (x) +   \int p(x) \, \mathrm{d} \mu^{(1)}  (x)  \\
   &+ t\, \alpha \, \left ( \frac{1}{1-\alpha} \, \int_{\mathbb{R}^D } p(x) \, \mathrm{d}  \mu^{(1)}  (x) - \frac{1}{\alpha} \, \int_{\mathbb{R}^D } p(x) \, \mathrm{d}\mu^{(0)}(x)  \right ), 
\end{align}
which by assumption (Equation~\eqref{eq:descent_case2}) and $0 < \alpha < 1 $ has derivative, 
\begin{equation}
   \frac{d}{dt} \int p (x)\, \mathrm{d} \pi(t)(x)  = \alpha \, \left (  \frac{1}{1-\alpha} \, \int_{\mathbb{R}^D } p(x) \, \mathrm{d} \mu^{(1)}(x) -  \frac{1}{\alpha} \, \int_{\mathbb{R}^D } p(x) \, \mathrm{d} \mu^{(0)}(x) \right ) < 0.
\end{equation}
This proves claim (2) and the statement in the second case.
\end{proof}

\begin{corollary}
    \label{cor:descentmeasure}
    Under the hypotheses of Proposition~\ref{prop:descentmeasure}, an element of $\mathcal{S}_2$ of the form,
    \begin{equation}
      \mu(\cdot) := \mu^{(0)}(\cdot) +   \mu^{(1)}(\cdot)  :=  \prod_{i=1}^D \mu^{(0)}_i(\cdot) +  \prod_{i=1}^D \mu^{(1)}_i(\cdot) 
    \end{equation} 
    is a global minimum of
    \begin{equation}
        \min_{ \mu(\cdot) \in \mathcal{S}_2 }\, \int p(x) \, \mathrm{d} \left (    \prod_{i=1}^D \mu_i^{(0)}(\cdot) +  \prod_{i=1}^D \mu^{(1)}_i(\cdot)   \right ) 
    \end{equation} 
    if and only if either (1) $\alpha^{(0)} := \prod_{i=1}^D \mu^{(0)}_i([-1,1]) \not = 0$ and $\mu(\cdot) = \frac{1}{\alpha^{(0)}} \, \prod_{i=1}^D \mu^{(0)}_i(\cdot)$ 
    %and $\prod_{i=1}^D \mu^{(1)}_i(\cdot) \equiv 0$,
    is itself a local minimum, or (2) $\alpha^{(1)} := \prod_{i=1}^D \mu^{(1)}_i([-1,1]) \not = 0$ and
    $\mu(\cdot) =\frac{1}{\alpha^{(1)}} \, \prod_{i=1}^D \mu^{(1)}_i(\cdot)$
    %and $\prod_{i=1}^D \mu^{(0)}_i(\cdot) \equiv 0$
    is itself a local minimum. 
    
Further, for each $l=0,1$ such that $\alpha^{(l)} \not = 0$, 
    \begin{equation*}
      \int p(x) \, \mathrm{d} \left ( \frac{1}{\alpha^{(l)}} \, \prod_{i=1}^D \mu^{(l)}_i(\cdot)   \right ) =  \min_{ \mu(\cdot) \in \mathcal{S}_2 }\,  \int p(x) \, \mathrm{d} \left (    \prod_{i=1}^D \mu_i^{(0)}(\cdot) +  \prod_{i=1}^D \mu^{(1)}_i(\cdot)   \right )
    \end{equation*}
    
\end{corollary}

\begin{proof}
    First, note that if either of $ \mu^{(0)} (\cdot) $ or $ \mu^{(1)} (\cdot) $ is feasible and is a local minimum, then we invoke Case (1) of Proposition~\ref{prop:descentmeasure} (i.e., with $\alpha^{(0)} \in \{0,1\}$) with the combination $\frac{1}{\alpha^{(0)}} \, \mu^{(0)} (\cdot) + 0  \in \mathcal{S}_2$ or $0 + \frac{1}{\alpha^{(1)}} \, \mu^{(1)} (\cdot) \in \mathcal{S}_2 $, respectively. This implies there exists a descent path from each point or that the measure is a global minimum. Since it is a local minimum by assumption, there cannot be a feasible descent path originating from that point, and it is therefore a global minimum.

    Next, assume the measure is a global minimum with value $p^*$. If $\mu^{(0)} \left( [-1,1]^D \right ) \in \{0,1\}$, the result follows trivially as point $\mu^{(0)}(\cdot)$ or $\mu^{(1)}(\cdot)$ must also be a non-trivial local minimum.
    
    Otherwise, $\mu^{(0)} \left( [-1,1]^D \right ) \in (0,1)$, and Proposition \ref{prop:descentmeasure} implies 
    \begin{equation}
        \label{eq:measure_equal_val}
        \frac{1}{\alpha^{(0)}} \, \int p(x) \, \mathrm{d} \prod_{i=1}^D \mu^{(0)}_i(x) =  \frac{1}{\alpha^{(1)}} \,\int p(x) \, \mathrm{d} \prod_{i=1}^D \mu^{(1)}_i(x) .
    \end{equation}
    Without an equality, this corresponds to Case (2) of Proposition~\ref{prop:descentmeasure}; therefore, there exists a descent path, a contradiction to local optimality. Equation~\eqref{eq:measure_equal_val} then implies 
    \begin{align}
        p^* &= \int p(x) \, \mathrm{d} \left (  \prod_{i=1}^D \mu^{(0)}_i(x)  +   \prod_{i=1}^D \mu^{(1)}_i(x)\right )  \, \mathrm{d} x = \frac{1}{\alpha^{(0)}} \, \int p(x) \, \mathrm{d} \prod_{i=1}^D \mu^{(0)}_i(x) , 
    \end{align}
    along with an analogous equation for $\frac{1}{\alpha^{(1)}} \, \prod_{i=1}^D \mu^{(1)}_i(x) $. Thus, both normalized product measures achieve the optimal value and must be global (and therefore local) minima, proving the claim.
\end{proof}

\globalmin*
\begin{proof}

To begin, assume that the feasible point,
$$ \left \{ \left (\mu^{(l)}_1, ..., \mu^{(l)}_D \right ) \right \}_{l=1}^2, $$ is a global minimum of Problem \ref{eq:reformulationcontinuous}.  Further, note that
\begin{equation}
   \{ \mu_{\bf n} \}_{\lvert {\bf n} \rvert \leq d }  := \left \{ \sum_{l=1}^2 \prod_{i=1}^D \mu^{(l)}_{i, {\bf n}_i} \right \}_{\lvert {\bf n} \rvert \leq d } \in \mathcal{F}_2
\end{equation} 
by construction. For $l=1,2$, let $\mu^{(l)}(\cdot)$ be a regular Borel product measure supported over $[-1,1]^D$ and associated with the moments $ \left (\mu^{(l)}_1, ..., \mu^{(l)}_D \right )$~\footnote{The existence of this measure and associated moments is guaranteed by Proposition \ref{prop:measureextension}}. In particular, it can be seen that
\begin{eqnarray}
    \mu(\cdot) := \mu^{(0)}(\cdot) + \mu^{(1)} (\cdot) \in  \mathcal{S}_2,
\end{eqnarray}
and 
\begin{equation}
    \zeta \left (  \mu(\cdot) \right ) =  \{ \mu_{\bf n} \}_{\lvert {\bf n} \rvert \leq d },
\end{equation}
where $\zeta : \mathcal{S}_2 \rightarrow \mathcal{F}_2$ is the surjective map of Proposition~\ref{prop:eqsets}. 
% Since both Problem \ref{eq:problemmeasure} and \ref{eq:problemmoments} share the same optimal value, we conclude that $ \mu(\cdot)$ must be a global minimum of Problem \ref{eq:problemmeasure} (both objective functionals have the same value).
% Invoking Corollary \ref{cor:equivproblem} (Case 3), we find that $\mu(\cdot)$ must be a local minimum of Problem \ref{eq:problemmeasure}. By Theorem \ref{}, however, we find that the value of the objective functional at that point equals the optimal 
Corollary \ref{cor:descentmeasure} then implies that the following holds for at least one $l \in \{0,1\}$: $\alpha^{(l)} \not = 0  $, and
\begin{equation}
    \frac{1}{\alpha^{(l)}} \, \prod_{i=1}^D \mu^{(l)}_i(\cdot)
\end{equation}
is a global minimum of Problem \ref{eq:problemmeasure} (it is a local minimum that reaches the globally optimal value). Finally, Corollary \ref{cor:equivproblem} (Case 2) implies 
\begin{equation}
    \zeta \left ( \frac{1}{\alpha^{(l)}} \, \prod_{i=1}^D \mu^{(l)}_i(\cdot) \right ) = \left \{ \frac{1}{\alpha^{(l)}} \prod_{i=1}^D \mu^{(l)}_{i, {\bf n}_i} \right \}_{\lvert {\bf n} \rvert \leq d }
\end{equation}
is a global (and therefore a local) minimum of Problem \ref{eq:problemmoments}. The pre-image of these moments under the map $\phi(\cdot)$ of Problem \ref{eq:reformulationcontinuous} contains the point,
    \begin{equation}
        \left \{  \left ( \frac{\mu_1^{(l)}}{\alpha^{(l)}}, \mu_2^{(l)} ..., \mu_D^{(l)} \right ) ,  (0,...,0)  \right \}.
    \end{equation}
    % for any combination of factors $0 < \beta^{(l)}_i \in \mathbb{R}$ such that $\prod_{i=1}^D \beta^{(l)}_i = \prod_{i=1}^D \mu^{(l)}_{i,0} =  \alpha^{(l)}$.
% These moments, however, may be constructed from any point of the form,
% \begin{equation}
%     \left ( \mu_1^{(l)}, ..., \mu_D^{(l)} \right )
% \end{equation}
% in Problem \ref{}. 
This demonstrates that the latter must correspond to a global minimum of Problem \ref{eq:problemcontinuous}, as claimed. For the converse, assume w.l.o.g. that
\begin{equation}
    \left \{  \left ( \frac{\mu_1^{(0)}}{\alpha^{(0)}}, ..., \mu_D^{(0)} \right ) ,  (0,...,0)  \right \}
\end{equation}
is a local minimum. Therefore,
\begin{equation}
   \left \{ \frac{\prod_{i=1}^D \mu_{i,{\bf n}_i}^{(0)}}{ \alpha^{(0)}} \right \} \in \mathcal{F}_2
\end{equation}
must be a local minimum of Problem \ref{eq:problemmoments} by continuity of the map $\phi(\cdot)$. Further, Proposition \ref{prop:measureextension} and Proposition \ref{prop:eqsets} guarantee the existence of a product measure,
$$
 \frac{\prod_{i=1}^D \mu_i^{(0)}(\cdot)}{ \alpha^{(0)}}  \in \mathcal{S}_2, $$
associated with these moments in the (non-empty) pre-image of $\zeta(\cdot)$.  Corollary~\ref{cor:equivproblem} (Case 3) implies the latter must be a
local minimum of Problem~\ref{eq:problemmeasure}. In turn, Corollary~\ref{cor:descentmeasure} implies this must actually be a global minimum of Problem \ref{eq:problemmeasure}. Finally, Corollary \ref{cor:equivproblem} (Case 2) implies the moments,
\begin{equation}
    \left \{ \left [ \zeta \left ( \frac{\mu^{(0)}(\cdot) }{\alpha^{(0)}} \right ) \right ]_{\bf n} \right \}_{\lvert {\bf n} \rvert \leq  d} = \left \{  \frac{ \prod_{i=1}^D \mu_{i,{\bf n}_i}^{(0)} }{\alpha^{(0)}}  \, \right \}_{\lvert {\bf n} \rvert \leq  d},
\end{equation}
must be associated with a global minimum of Problem \ref{eq:problemmoments}. These moments correspond to the image of the point under consideration under the map $\phi(\cdot)$, which demonstrates that the latter is a global minimum of Problem~\ref{eq:reformulationcontinuous}.
\end{proof}

% \section{Notation Table}
% \label{sec:notation_table}
% \begin{table}[htbp!]
% {\renewcommand{\arraystretch}{1.5}%
%     \begin{tabular}{|c|l|} \hline
%         {\bf Symbol} & {\bf Definition}   \\ \hline
%         $D \in \mathbb{N}$ & Dimension of Problem~\ref{eq:problemcontinuous}  \\ \hline
%         $d$ & Problem degree (i.e., degree of objective polynomial)  \\ \hline
%         $n=(n_1, n_2, ..., n_D)\in \mathbb{N}^D$ & Multi-index \\ \hline
%         $\lvert n \rvert = \sum_{i=1}^D n_i$ &
%         Multi-index degree \\ \hline
%         $x^n = \prod_{i=1}^D x_i^{{\bf n_i}}$ &
%         Basis of monomials \\ \hline
%         \multirow{2}*{$\mathbb{P}_{D,d}$} & 
%         \multirow{2}{7cm}{Vector space of degree-$d$ polynomials in dimension $D$ (Equation~\eqref{eq:P_Dd})} \\\\ \hline
%         Test & Test \\ \hline
%     \end{tabular}}
%     \caption{Blerp}
%     \label{tab:notation_appendix}
% \end{table}

\section{Algorithm}
\label{sec:algorithm}
This section details algorithmic implementation specifics and computational considerations for the numerically solving our re-formulation (Problem~\ref{eq:reformulationcontinuous}). This algorithm and its implementation were used to produce the results in Section~\ref{sec:numerical_results}.~\footnote{We reiterate that the user by no way needs to use this algorithmic implementation in order to solve the reformulated problem; any algorithm or its implementation suitable for computing a local minimum of the nonlinear Problem~\ref{eq:reformulationcontinuous} is appropriate. }
% Theorem~\ref{thm:equivalence} demonstrates the feasibility of recovering the optimal value of Problem~\ref{eq:problemcontinuous} by solving Problem~\ref{eq:reformulationcontinuous}. In addition, Theorem~\ref{thm:globalmin} dictates that such optimal values can be reached using local descent methods only. Based on these theoretical results and techniques, we have implemented a numerical algorithm for solving the reformulated problem. 
% The remainder of this section provides details about our algorithm.\footnote{There exists many alternative algorithms for solving Problem~\ref{eq:reformulationcontinuous}. Our algorithm is a single instance and represents our attempt at devising an efficient way to solve the reformulated problem. It is by no means unique.} Further, technical details specific to our implementation can be found in Section~\ref{sec:algdetails}, and numerical results are discussed in Section~\ref{sec:numerical_results}.

Our algorithm is an extension of the method of Burer and Monteiro (BM), presented in \cite{burer2003nonlinear}. The BM method replaces positive-definite constraints of the form, $A \succeq 0$, with constraints of the form, $A = R R^T$, for some matrix $R$ of appropriate size. Problem~\ref{eq:problemcontinuous} can thus be written as
% ~\footnote{In order to simplify our implementation, we fix $\mu^{(l)}_{i,0} = 1$ for $i\geq 2$, while allowing $\mu^{(l)}_{0,0}$ to remain variable.}
\begin{align}
    \label{eq:reformulationcontinuous_BM}
    \min_{  \mu \in \mathbb{R}^{  (d+1) \times D \times L}   }&\;\; \sum_{{\bf n} \in \mathrm{supp}(p)} \, p_{\bf n} \, \phi_{\bf n} (\mu),\nonumber \\
    \mathrm{subject\, to} & \;\; \mathcal{M}_d(\mu^{(l)}_i) = R^{(l)}_i R^{(l)^T}_i ,  \\
    &\;\; \mathcal{M}_{d-1}(\mu^{(l)}_i; \, 1-x_i^2) = S^{(l)}_i S^{(l)^T}_i , \; & i=1,..., D, l=1,..., L \nonumber\\
    &\;\; R_{i}^{(l)} \in \mathbb{R}^{(d+1) \times (d+1)} , \, S_{i}^{(l)}  \in \mathbb{R}^{ d \times d }, \nonumber\\
    & \; \; \mu^{(l)}_{1,0} \geq 0, \nonumber \\
    & \; \; \mu^{(l)}_{i,0} = 1, i = 2, ..., D \nonumber \\
    & \;\;   \phi_{\bf n} (\mu) =  \sum_{l=1}^L  \prod_{i=1}^D \mu^{(l)}_{i,{\bf n}_i} .\nonumber
\end{align}

Upon re-expressing the (element-wise) matrix equality constraints as a sequence of $K = L\cdot D ( (d+1)^2 + d^2 + 1) + 1$ scalar equality constraints, we obtain the generalized BM problem reformulation,
\begin{align}
    \label{eq:reformulationcontinuous_BM_general}
    \min_{ \mu \in \mathbb{R}^{  (d+1) \times D \times L}   }&\;\; \sum_{{\bf n} \in \mathrm{supp}(p)} \, p_{\bf n} \, \phi_{\bf n}(\mu) \\
    %\left ( \sum_{l=1}^L  \prod_{i=1}^D \mu^{(l)}_{i,{\bf n}_i} \right )\\
    \mathrm{subject\, to} &\;\; c_k\left ( \mu, \{R_i^{(l)}, S_i^{(l)}\} \right ) = 0 , &\;\;  \, k=1,..., K. \nonumber
\end{align}
We consider this form in Equation~\eqref{eq:reformulationcontinuous_BM_general} for numerical optimization.  As proposed in~\cite{burer2003nonlinear}, our numerical implementation constructs the augmented Lagrangian for this problem as 
\begin{align}
    \label{eq:augmented_lagrangian}
    \mathcal{L}_\gamma (\mu, \{R_i^{(l)}, S_i^{(l)}\}; \lambda) &:=  \sum_{{\bf n} \in \mathrm{supp}(p)} \, p_{\bf n} \, \phi_{\bf n}(\mu) +  \sum_{k=1}^K \lambda_k \cdot c_k\left ( \mu, \{R_i^{(l)}, S_i^{(l)}\} \right ) \\
    &+ \frac{\gamma}{2} \, \sum_{k=1}^K  c^2_k\left ( \mu, \{R^{(l)}, S^{(l)}\} \right ), \nonumber
\end{align}
where $\gamma > 0$ is a penalty parameter. This is similar to the Lagrangian introduced in~\cite{burer2003nonlinear}, with the difference that the objective exhibits non-linearities.  We solve~\ref{eq:augmented_lagrangian} numerically using a technique similar to one proposed by Boyd et al.~\cite{boyd2011distributed}, which involves finding a saddle-point (i.e., a simultaneous minimum and maximum in the primal and dual variables respectively) to the augmented Lagrangian.  Specifically, our approach focuses on finding $\mu^*$ and $\lambda^*$ such that
\begin{equation}
    \mathcal{L}_\gamma (\mu^*, \{R_i^{(l)^*}, S_i^{(l)^*}\}; \lambda^*) = \max_\lambda \,\min_{\mu, \{R_i^{(l)}, S_i^{(l)}\}}   \mathcal{L}_\gamma (\mu, \{R_i^{(l)}, S_i^{(l)}\} ; \lambda).
\end{equation}
To this end, our implementation alternates solving the \emph{primal} and \emph{dual} problems until convergence is achieved. The {\bf primal problem takes the form,}
\begin{equation}
    \min_{\mu, \{R_i^{(l)}, S_i^{(l)}\}} \, \mathcal{L}_\gamma (\mu, \{R^{(l)}, S^{(l)}\}; \lambda),
\end{equation}
where the Lagrange multipliers (dual variables) $\lambda$ are fixed.  Thanks to the penalty term, this is an unconstrained optimization problem with an objective that is bounded below.  This can be solved using a descent technique; for this purpose, our algorithm leverages L-BFGS \cite{liu1989limited}.  Meanwhile, the {\bf dual problem takes the form,}
\begin{equation}
    \max_{\lambda} \, \mathcal{L}_\gamma (\mu, \{R^{(l)}, S^{(l)}\}; \lambda),
\end{equation}
where the primal variables $\mu$,  $R^{(l)}$  and $S^{(l)}$ are fixed. Because this problem is generally unbounded at non-optimal points, we use an update rule proposed by \cite{boyd2011distributed}:
\begin{eqnarray}
    \lambda_{\mathrm{next}} &=& \lambda_{\mathrm{current}} + \gamma \, \nabla_{\lambda} \,\mathcal{L}_\gamma (\mu, \{R_i^{(l)}, S_i^{(l)}\}; \lambda)\\
    &=& \lambda_{\mathrm{current}} + \gamma \, \sum_{k=1}^K  c^2_k(\mu, \{R_i^{(l)}, S_i^{(l)}\}).
    \nonumber
\end{eqnarray}
This process is summarized in Algorithm \ref{alg:alternating_saddle_point}. We have implemented this algorithm in \texttt{C/C++} with numerical results are presented in Section~\ref{sec:numerical_results}.

\begin{algorithm} 
\small
\caption{\footnotesize{Alternating descent/ascent algorithm for computing a saddle-point of the augmented Lagrangian}}
\label{alg:alternating_saddle_point} 
\begin{algorithmic}[1]
    \Procedure{AlternatingSaddlePoint}{$\mu_0, \lambda_0; \gamma$}
        \State $\mu \gets \mu_0$
        \State $\lambda \gets \lambda_0$
        \While{Not Converged} %\Comment{We have the answer if r is 0}
            \State $\mu \gets \mathrm{argmin}_{\mu} \, \mathcal{L}_\gamma (\mu; \lambda)$\Comment{Using L-BFGS}
            \State $\lambda \gets \lambda + \gamma \, \nabla_{\lambda} \,\mathcal{L}_\gamma (\mu; \lambda) $
        \EndWhile\label{euclidendwhile}
        \State \textbf{return} $(\mu^*,\lambda^*)$
    \EndProcedure
\end{algorithmic}
\end{algorithm}

\subsection{Optimal Location}
\label{sec:optimal_location}
When the global minimum $x^* \in [-1,1]^D$ of Problem \ref{eq:problemcontinuous} is unique, it can be shown that the solution of the reformulated problem (Problem \ref{eq:reformulationcontinuous}) must consist of the (redundant) convex combinations of delta measures centered at the point $x^*$. Under these conditions, it is found that
\begin{equation}
    \frac{\mu^{(l)}_{i,{\bf n}_i} }{\prod_{i=1}^D \mu_{i,0}^{(l)} } = \int_{-1}^1 x_i^{{\bf n}_i} \,  \delta(x_i - x_i^*) \, \mathrm{d} x_i = (x_i^*)^{{\bf n}_i},
\end{equation}
for every fixed index $l$ such that $ \prod_{i=1}^D, \mu_{i,0}^{(l)} \not = 0 $.  In this special case, the above scheme does not merely provide a framework for finding the global optimal value, but it also provides a simple way of computing a globally-optimal location. Indeed, it suffices to pick any $l$ which satisfies the above condition and compute
\begin{equation}
    x_i^* = \frac{\mu^{(l)}_{i,1} }{ \prod_{i=1}^D  \mu_{i,0}^{(l)} },\quad{   }\forall i = 1, ..., D.
\end{equation}
\\This approach is summarized in Algorithm \ref{alg:solution_extraction}.\footnote{When the global minimum is not unique (or not known to be so), a small amount of noise is added to the objective, making it unique with high probability.} 

\begin{algorithm}
\small
\caption{\footnotesize{Optimal location extraction (assuming unique global minimum)}}
\label{alg:solution_extraction} 
\begin{algorithmic}[1]
    \Procedure{OptimalLocation}{$\mu$}\Comment{Input is output of Algorithm \ref{alg:alternating_saddle_point} }
        \State$l^* \gets \mathrm{argmax}_{l=1,...,L} \, \prod_{i=1}^D \mu_{i,0}^{(l)} $
        \For{$i = 1, ..., D$} 
            \State $ x_i^* \gets \frac{\mu^{(l^*)}_{i,1} }{\prod_{i=1}^D \mu_{i,0}^{(l^*)} } $
        \EndFor
        \State \textbf{return} $x^*$\Comment{Global minimum location}
    \EndProcedure
\end{algorithmic}
\end{algorithm}

\subsection{Further Implementation details}
\label{sec:algdetails}
Based on the high-level algorithmic information provided in the previous sections, this section provides further details and implementation technicalities, 

\subsubsection{Constraint simplification}
As implied in Problem~\ref{eq:reformulationcontinuous}, our reformulation dictates that the semi-definite constraints should be expressed as
\begin{align*}
    &\mathcal{M}_d(\mu^{(l)}_i) \succeq 0 \quad{and} \\
    & \mathcal{M}_{d-1}(\mu^{(l)}_i; \, 1-x_i^2) \succeq 0,
\end{align*}
for $l=1, ..., L$ and $i=1,...,D$. Within our implementation, however, we use the constraints,
\begin{align*}
    &\mathcal{M}_d(\mu^{(l)}_i) \succeq 0, \; \lvert \mu^{(l)}_{i,{\bf n}_i} \rvert \leq 1,
\end{align*}
for $i=1,...,D $ and $0 \leq {\bf n}_i \leq d$. This reduces computational complexity and memory costs by a factor of roughly $2$, simplifying the implementation. This also constitutes a relaxation of the constraint: $\mathcal{M}_{d-1}(\mu^{(l)}_i; \, 1-x_i^2) \succeq 0$. However, this does not materially change the conclusion of our theoretical analysis thanks to the following lemma:

\begin{lemma}
    Assume that $2L > d \geq 0$ and assume that the 1D moments $\{\mu_{k} \}_{k=0}^{2L}$ are such that $\mathcal{M}_L(\mu) \succeq 0$ and satisfy,
    \begin{equation}
        \lvert \mu_k \rvert \leq 1
    \end{equation}
    for every $0 \leq k \leq 2L$. Then, the following holds for every polynomial $p \in \mathbb{P}_{D,d}$:
    \begin{eqnarray}
        \sum_{j=0}^d p_j \, \mu_j = \int p(x) \, \mathrm{d} \mu(x) = \int_{[-1-\alpha, 1+\alpha]} p(x) \, \mathrm{d} \mu(x)  + \epsilon(\alpha, L),
    \end{eqnarray}
    where
    \begin{equation}
        \lvert \epsilon(\alpha, L) \rvert \leq (d+1)\, \frac{\max_{j=0}^d \lvert p_j \rvert  }{ \alpha \, (1+\alpha)^{2L-(d+1)}   } = O \left ( \frac{1}{\alpha \, (1 + \alpha)^{2L-(d+1)} } \right )
    \end{equation}
    for every $\alpha >0$.
\end{lemma}
\begin{proof}
To begin, recall that in 1D the condition $\mathcal{M}_L(\mu) \succeq 0$ guarantees the existence (Gaussian quadratures, see, e.g., \cite{golub2009matrices, szeg1939orthogonal}) of non-negative weights $\{ w_i \}_{i=0}^d$ and locations $\{ x_i \}_{i=0}^d$ such that
\begin{eqnarray}
    \int p(x) \, \mathrm{d} \mu(x) = \sum_{i=0}^d w_i \, p(x_i) = \int p(x) \, \left ( \sum_{i=0}^d w_i \, \delta(x-x_i) \right) \,\mathrm{d}x.
\end{eqnarray}
In particular, note that
\begin{equation}
   0 \leq \mu_{2L} = \int x^{2L} \, \mathrm{d} \mu(x) = \sum_{i=1}^R w_i \, x_i^{2L} \leq 1.
\end{equation}
Let $\alpha > 0$. Because all quantities involved are positive, it must hold that
\begin{eqnarray}
    w_i \leq \left\{
	\begin{array}{ll}
		1  & \mbox{if } x_i \in [-1-\alpha, 1 + \alpha] \\
		\frac{1}{x_i^{2L}} & \mbox{if } \lvert x_i \rvert > 1 + \alpha.
	\end{array}
\right.
\end{eqnarray}
Therefore,
\begin{align}
    \left \lvert \int_{\vert x \rvert \geq 1+\delta} p(x) \, \mathrm{d} \mu(x) \right \rvert &= \left \lvert \sum_{i : \lvert x_i \rvert \geq 1+\delta } w_i \,\sum_{j=0}^d  p_j x_i^j \right \rvert \\
    &\leq \sum_{i : \lvert x_i \rvert \geq 1+\delta } \,\sum_{j=0}^d  \lvert p_j \rvert \,  \frac{1}{\lvert x_i \rvert^{2L-j} } \\
    & \leq (d+1) \, \frac{\max_{j=0}^d \lvert p_j \rvert  }{ (1+\alpha)^{2L}  } \, \sum_{j=0}^d (1+\alpha)^j \\
    &= (d+1) \, \frac{\max_{j=0}^d \lvert p_j \rvert }{ (1+\alpha)^{2L}  } \, \frac{(1+\alpha)^{d+1} - 1}{\alpha} \\
    &\leq (d+1)\, \frac{\max_{j=0}^d \lvert p_j \rvert  }{ \alpha \, (1+\alpha)^{2L-(d+1)}.  }
\end{align}
This proves the desired result.
\end{proof}
The conclusion of the lemma also holds for product measures in higher dimension $D$ with an additional factor of $D$ in the bound. Effectively, by increasing the number of known moments of the 1D measures, one can ensure there exists a measure with such moments approximately~\footnote{In the above functional sense.} supported over $[-1-\delta, 1+\delta]^D$. Using $L=d$ was sufficient to obtain satisfactory accuracy in our experiments.

Finally, the scalar inequality constraints were enforced using the BM method as described earlier. A BM factor $R_k R_k^T$ is trivial and corresponds to a squared scalar. For instance, the constraint $\mu^{(l)}_{i,{\bf n}_i} \leq 1$ becomes
$$1-\mu^{(l)}_{i,{\bf n}_i}  = (t^{(l)}_{i,{\bf n}_i})^2, \;\; t^{(l)}_{i,{\bf n}_i} \in \mathbb{R}.$$

\subsubsection{Numerical stability}
The aforementioned constraints are sufficient to ensure the moments of the product measures are bounded within $[-1,1]$; i.e.,
\begin{equation}
    \label{eq:redundant_constraint}
   -1 \leq   \prod_{i=1}^D \mu^{(l)}_{i,{\bf n}_i} \leq 1
\end{equation}
for $l=1,...,L$ and ${\bf n} \in \mathrm{supp}(p)$. It may occur, however, that such constraints are violated, resulting in large values for the product moments and potential numerical instabilities. This is of particular concern in higher dimensions, where small violations may compound to create large values through multiplication.  In order to mitigate this potential numerical instability and improve the numerical behavior, our implementation \emph{explicitly enforces the (redundant) constraints} found in Equation~\eqref{eq:redundant_constraint} by adding them to the formulation found in Equation~\eqref{eq:reformulationcontinuous_BM_general}. 

\subsubsection{Polynomial basis}
\label{subsec:polybasis}
Our choice of monomials as a basis for the polynomials in Section \ref{sec:notation} is made to facilitate the exposition; however, any polynomial basis may have been used. In our implementation, we leverage a \emph{Chebyshev basis} $\{T_{{\bf n}_i}(x_i) \}$ (e.g.,\cite{mason2002chebyshev}).  This choice offers better numerical stability without significantly increasing computational complexity. The only differences are that the 1D moments are defined as
\begin{equation}
    \mu_{i,{\bf n}_i} := \int_{-1}^1 \, T_{{\bf n}_i}(x_i) \, \mathrm{d} \mu_i (x_i),
\end{equation}
and the moment matrices $ \mathcal{M}_d(\mu_i)$ have the (Hankel plus Toeplitz) form,
\begin{equation}
     [ \mathcal{M}_d(\mu_i) ]_{{\bf m}_i,{\bf {\bf n}}_i} = \int_{\mathbb{R}} \, T_{{\bf m}_i}(x) \,  T_{{\bf n}_i} (x_i) \, \mathrm{d} \mu_i(x_i) =  \frac{1}{2} \, \mu_{{\bf m}_i+{\bf n}_i} + \frac{1}{2} \, \mu_{\lvert {\bf m}_i - {\bf n}_i \rvert}.
\end{equation}
This is a consequence of the fact that Chebyshev polynomials satisfy the identity,
\begin{equation}
    T_{{\bf m}_i}(x_i) \,  T_{{\bf n}_i}(x_i) = T_{{\bf m}_i + {\bf n}_i}(x_i) + T_{\lvert {\bf m}_i - {\bf n}_i \rvert} (x_i).
\end{equation}

\subsubsection{Initialization}
Our current implementation uses a simple initialization scheme (no warm start) for the primal and dual variables of the augmented Lagrangian (Equation~\eqref{eq:augmented_lagrangian}):

\begin{itemize}
    \item \emph{Primal Initialization}: each scalar element of the primal variables, $\mu , \, \{R_i^{(l)}, S_i^{(l)}\}$,
    is initialized randomly using a distribution $U\left ( [-1,1] \right )$ uniform over $[-1,1]$.
    \item \emph{Dual Initialization}: the dual variables ($\lambda$) are initialized to be zero ($0$).
\end{itemize}

\subsubsection{L-BFGS and Line Search}
As previously discussed, we use L-BFGS \cite{liu1989limited} to compute descent directions (${\bf g}$) at each iteration when solving the primal problem. Iterates are computed using a line search with backtracking and a Wolfe condition; i.e., the scheme computes the next point that corresponds to the smallest $k$, such that
\begin{equation}
    f({\bf x}_{\mathrm{current}} + 0.5^k \, {\bf d}) \leq  f({\bf x}_{\mathrm{current}}) + 0.5^k \, \beta \, ( \nabla f({\bf x}_{\mathrm{current}}) \cdot {\bf g })
\end{equation}
for some parameter $0 < \beta < 1$ (referred to as the Wolfe factor). The order of approximation of the Hessian typically ranges between $20$ and $200$. We further leverage a mechanism by which the algorithm reverts to a gradient descent direction whenever the line search fails using the direction ${\bf g}$ computed by L-BFGS.

\subsubsection{Stopping criteria}
We use two (2) types of stopping criteria, both for the inner (based on L-BFGS)  and outer (saddle-point) iterations: 
\begin{enumerate}
    \item A criterion based on the size of the relative gradient; 
    \item A criterion based on the absolute change in value between the previous and current iterates. 
\end{enumerate}

As a part of the stopping criterion for the outer iteration, we further enforce feasibility; i.e., 
$$ \lvert c_k (x) \rvert < \epsilon_{f},$$
for some {\em absolute feasibility} stopping parameter $0 < \epsilon_f < 1$.
%\end{appendices}


\begin{thebibliography}{9}
\bibitem{LAURENT2005393} Aardal, K. and  G.L. Nemhauser and R. Weismantel.
``Semidefinite Programming and Integer Programming.'' \textit{Handbooks in Operations Research and Management Science} 12 (2005): 393-514.
\bibitem{andersen2011interior} Andersen, Martin, et al. ``Interior-point methods for large-scale cone programming.'' \textit{Optimization for machine learning} 5583 (2011).
\bibitem{ben2001lectures} Ben-Tal, Aharon, and Arkadi Nemirovski. \textit{Lectures on modern convex optimization: analysis, algorithms, and engineering applications}. Society for industrial and applied mathematics, 2001.
\bibitem{boumal2016non} Boumal, Nicolas, Vlad Voroninski, and Afonso Bandeira. ``The non-convex Burer-Monteiro approach works on smooth semidefinite programs.'' \textit{Advances in Neural Information Processing Systems} 29 (2016).
\bibitem{boyd2011distributed} Boyd, Stephen, et al. ``Distributed optimization and statistical learning via the alternating direction method of multipliers.'' \textit{Foundations and Trends in Machine learning} 3.1 (2011): 1-122.
\bibitem{Bubeck} Bubeck, Sebastien. ``Convex Optimization: Algorithms and Complexity.'' \textit{arXiv:1405.4980v2}. (2014)
\bibitem{burer2003nonlinear} Burer, Samuel, and Renato DC Monteiro. ``A nonlinear programming algorithm for solving semidefinite programs via low-rank factorization.'' \textit{Mathematical Programming} 95.2 (2003): 329-357.
\bibitem{cohn2013measure} Cohn, Donald L. \textit{Measure theory}. Vol. 1. New York: Birkhauser, 2013.
\bibitem{Klerk2011} 
de Klerk, Etienne and Monique Laurent. ``On the Lasserre Hierarchy of Semidefinite Programming Relaxations of Convex Polynomial Optimization Problems.'' \textit{SIAM Journal on Optimization} 21 (2011):824-832. 
\bibitem{curto1991recursiveness} Raul E Curto. ``Recursiveness, positivity and truncated moment problems.'' \textit{Houston Journal of Mathematics}, 17:603–635, 1991.
\bibitem{de2022convergence} de Klerk, Etienne, and Monique Laurent. ``Convergence analysis of a Lasserre hierarchy of upper bounds for polynomial minimization on the sphere.'' \textit{Mathematical Programming} 193.2 (2022): 665-685.
\bibitem{Dressler2017} 
Dressler, Mareike, Sadik Iliman, and Timo de Wolff. ``A Positivstellensatz for Sums of Nonnegative Circuit Polynomials.'' \textit{SIAM Journal of Applied Algebraic Geometry} vol 1 (2017):536-555.
\bibitem{Fujisawa}
Fujisawa, K., Fukuda, M., Kojima, M., Nakata, K. ``Numerical Evaluation of SDPA (Semidefinite Programming Algorithm).'' \textit{In: Frenk, H., Roos, K., Terlaky, T., Zhang, S. (eds) High Performance Optimization. Applied Optimization} vol 33. Springer, Boston, MA. 
\bibitem{golub2013matrix} Golub, Gene H., and Charles F. Van Loan. Matrix computations. JHU press, 2013.
\bibitem{golub2009matrices} Golub, Gene H. and Meurant, G{\'e}rard. \textit{Matrices, moments and quadrature with applications} : 30. Princeton University Press, 2009.
\bibitem{grant2011cvx} Grant, Michael, Stephen Boyd, and Yinyu Ye. ``CVX: Matlab software for disciplined convex programming.'' (2011).
\bibitem{KANNAN2019351} Kannan, Hariprasad and Nikos Komodakis and Nikos Paragios. ``Chapter 9 - Tighter continuous relaxations for MAP inference in discrete MRFs: A survey.'' \textit{Handbook of Numerical Analysis: Processing, Analyzing and Learning of Images, Shapes, and Forms: Part 2 (editors: Ron Kimmel and Xue-Cheng Tai)}
Elsevier, Vol. 20 (2019): 351-400.
\bibitem{lasserre2001global} Lasserre, Jean B. ``Global optimization with polynomials and the problem of moments.'' \textit{SIAM Journal on optimization} 11.3 (2001): 796-817.
\bibitem{doi:10.1137/040614141} 
Lasserre, Jean B. ``Sum of Squares Approximation of Polynomials, Nonnegative on a Real Algebraic Set.'' \textit{SIAM Journal on Optimization} 16.2 (2005): 610-628.
\bibitem{lasserre2009moments} Lasserre, Jean B. ``Moments and sums of squares for polynomial optimization and related problems.'' \textit{Journal of Global Optimization} 45.1 (2009): 39-61.
\bibitem{liu1989limited} Liu, Dong C., and Jorge Nocedal. ``On the limited memory BFGS method for large scale optimization.'' \textit{Mathematical programming} 45.1-3 (1989): 503-528.
\bibitem{mason2002chebyshev} Mason, John C., and David C. Handscomb. \textit{Chebyshev polynomials}. CRC press, 2002.
\bibitem{Mittleman} Mittelmann, H. ``An independent benchmarking of SDP and SOCP solvers.'' \textit{Mathematical Programming} Ser. B 95, (2003):4017-430.
\bibitem{DBLP:books/daglib/0071613}
Nesterov, Yurii E. and Arkadii Nemirovskii. ``Interior-point polynomial algorithms in convex programming.'' \textit{SIAM studies in applied mathematics} 13 (1994).
\bibitem{nie2023moment} Nie, J. (2023). ``Moment and Polynomial Optimization''. \textit{Society for Industrial and Applied Mathematics}, 2023.
\bibitem{nie2007complexity} Nie, Jiawang, and Markus Schweighofer. ``On the complexity of Putinar's Positivstellensatz.'' \textit{Journal of Complexity} 23.1 (2007): 135-150.
\bibitem{Parrilo2003} Parrilo, Pablo A. ``Semidefinite programming relaxations for semialgebraic problems.'' \textit{Mathematical Programming} 96:2 (2003):293-320.
\bibitem{powers2000polynomials} Powers, Victoria, and Bruce Reznick. "Polynomials that are positive on an interval." \textit{Transactions of the American Mathematical Society} 352.10 (2000): 4677-4692.
\bibitem{prajna2002introducing} Prajna, Stephen and Papachristodoulou, Antonis and Parrilo, Pablo A. ``Introducing SOSTOOLS: A general purpose sum of squares programming solver.'', \textit{Proceedings of the 41st IEEE Conference on Decision and Control} 1 (2002): 741-746
\bibitem{putinar1993positive} Putinar, Mihai. ``Positive polynomials on compact semi-algebraic sets.'' \textit{Indiana University Mathematics Journal} 42.3 (1993): 969-984.
\bibitem{reed1980functional} Reed, Michael, and Barry Simon. \textit{I: Functional analysis}. Vol. 1. Gulf Professional Publishing, 1980.
\bibitem{theodore1965motzkin} Theodore, S. ``Motzkin. The arithmetic-geometric inequality.'' Inequalities (\textit{Proc. Sympos. Wright-Patterson Air Force Base, Ohio, 1965}): 205-224.
\bibitem{rudin1986real} Rudin, Walter. \textit{Real and complex analysis: 3rd. Ed.}. McGraw-Hill, 1987.
\bibitem{szeg1939orthogonal} Szego, Gabor. \textit{Orthogonal polynomials} : 23. American Mathematical Soc., 1939.
\end{thebibliography}
\end{document}